\documentclass{article}

\usepackage{frcursive}
\usepackage{cancel}
\usepackage{amsfonts,amssymb}
\usepackage[latin1]{inputenc}
 \usepackage[T1]{fontenc}
\usepackage{tikz}
\usetikzlibrary{matrix}

\usepackage{amsmath}

\def\eg{e.g.}

\def\mcM{{\mathcal M}}

\def\Hom{{\rm Hom}}
\def\Mor{{\rm Mor}}
\def\Hom{{\rm Hom}}

\def\lg{{\mathfrak g}}

\def\lh{{\mathfrak h}}\def\lk{{\mathfrak k}}\def\la{{\mathfrak a}}

\def\lh{{\mathfrak h}}
\def\CC{{\mathbb C}}\def\PP{{\mathbb P}}\def\ZZ{{\mathbb
    Z}}\def\RR{{\mathbb R}}
\def\QQ{{\mathbb Q}}

\def\hlambda{{\hat\lambda}}

\def\longto{\longrightarrow}

\def\Orb{{\mathcal O}}

\def\Li{{\mathcal L}}

\def\alc{{\mathcal A}}
\def\Lie{{\rm Lie}}
\def\Pic{{\rm Pic}}

\def\iI{{\mathcal I}}

\def\kmV{{\mathcal H}}
\def\kmL{{\mathcal L}}

\def\bl{\text{\cursive l}}

\newtheorem{lemma}{Lemma}

\newtheorem{prop}{Proposition}
\newtheorem{theo}{Theorem}
\newtheorem{coro}{Corollary}

\newenvironment{proof}{{\noindent\bf Proof.}}{\hfill $\square$}
\newenvironment{defin}{{\noindent\bf Definition.}}{}
\newenvironment{NB}{{\noindent\bf Remark.}}{}

\def\Chi{{\mathcal X}}
\def\Gra{{\mathcal{G}}}
\def\kmG{{\bf G}}

\def\bccup{\bigsqcup}
\def\Iw{{\mathcal B}}
\def\Waff{{\tilde W}}
\def\Eps{{\mathcal E}}

\def\ud{{\underline d}}
\def\Md{{\overline{M}_{0,3}(G/P,\ud)}}
\def\ev{{\rm ev}}
\def\mvar{{\mathbb X}}\def\monx{{\mathbb x}}\def\monz{{\mathbb
    z}}\def\mony{{\mathbb y}}
\def\parE{{\mathcal E}}
\def\tiota{{\tilde\iota}}
\def\pardeg{{\rm pardeg}}

\def\Lila{{\Li(\bl\Lambda)\otimes \Li(\lambda_1)\otimes \Li(\lambda_2)\otimes \Li(\lambda_3)}}
\def\Face{{\mathcal F}}
\def\poly{{\mathcal P}}
\def\cone{{\mathcal C}}

\begin{document}
\title{On the quantum Horn problem}

\author{N. Ressayre\footnote{Universit{\'e} Lyon 1 - 
43, bd du 11 novembre 1918 -
69622 Villeurbanne Cedex-
France - {\tt ressayre@math.univ-lyon1.fr}}}

\maketitle

\begin{abstract}
Let $K$ be a compact, connected, simply-connected simple Lie group.
Given two   conjugacy classes $\Orb_1$ and $\Orb_2$ in $K$, we
consider  the multiplicative Horn question:
  What conjugacy classes are contained in $\Orb_1\cdot\Orb_2$?
It is known that answering this question remains to describe a convex
polytope $\poly_K$. In 2003, Teleman-Woodward gave a complete list of
inequalities for $\poly_K$. Their list contains redundant
inequalities. In this paper, we describe $\poly_K$ by a smaller list
of inequalities.
\end{abstract}

\noindent
{\bf Warning.}
During the redaction of this paper, Belkale-Kumar independently obtained in
\cite{BK:qHorn} similar results.   Moreover, it is proved in
\cite{BK:qHorn} that the list of inequalities obtained in
\cite{BK:qHorn} and here is irredundant.

\section{Introduction}

\subsection{The additive Horn problem}

Let $K$ be a compact, connected, simply-connected simple Lie group and
let $\lk$ denote its Lie algebra.
Let $\Orb_1$ and $\Orb_2$ be two adjoint $K$-orbit in $\lk$. Then the
sum $\Orb_1+\Orb_2=\{\xi_1+\xi_2\,:\, \xi_1\in\Orb_1$ and $\xi_2\in
\Orb_2\}$ is $K$-stable. 
The so called Horn question is:

\begin{center}
  What adjoint $K$-orbits are contained in $\Orb_1+\Orb_2$?
\end{center}

\noindent{\bf Parametrization of adjoint orbits.}
Let $G$ denote the complexification of $K$.
Fix a maximal torus $T$ of $G$ such that $T_K:=T\cap K$ is a maximal
torus of $K$. 
Any root $\alpha$ of $(G,T)$ induces (by derivation) a linear form
(still denoted by $\alpha$) on
the Lie algebra $\Lie(T)$ of $T$.
The Lie algebra $\Lie(T_K)$ of $T_K$ identifies with the real Lie
subalgebra of $\xi\in \Lie(T)$ such that $\alpha(\xi)\in \sqrt{-1}\RR$
for any root $\alpha$.

Let $X_*(T)$ denote the group of one parameter subgroups of $T$.
It identifies with a sublattice of $\Lie(T)$. Moreover, the spanned
real vector space $X_*(T)_\RR:=X_*(T)\otimes \RR$ is the real Lie
subalgebra of $\xi\in \Lie(T)$ such that $\alpha(\xi)\in \RR$
for any root $\alpha$.

Choose a Borel subgroup $B$ of $G$ containing $T$.
Let $\Delta$ denote the associated set of simple roots.
The dominant chamber  in $X_*(T)_\RR$ is 
$$
X_*(T)_\RR^+=\{\tau\in X_*(T)_\RR\,:\,
    \langle\tau,\alpha\rangle\geq 0\quad\forall\alpha\in\Delta
\}.
$$
Any adjoint $K$-orbit in $\Lie(K)$ contains a unique element of the form
$\sqrt{-1}\tau$ for some $\tau\in X_*(T)_\RR^+$ 
; we denote by
$\Orb_\tau$ the adjoint $K$-orbit containing $\sqrt{-1}\tau$. \\

\noindent{\bf The Horn cone.}
Answering the Horn question is equivalent of describing
the set
$$
\poly_{\Lie(K)}=\{(\tau_1,\tau_2,\tau_3)\in X_*(T)_\RR^+)^3\,:\, \Orb_{\tau_1}+
\Orb_{\tau_2}+ \Orb_{\tau_3}\ni 0\}.
$$
According to Kirwan's convexity theorem
\cite{Ki}, $\poly_{\Lie(K)}$ is a convex polytope of nonempty interior 
in $X_*(T)_\RR$.
Belkale-Kumar \cite{BK} obtained an explicit list of
inequalities that characterize $\poly_{\Lie(K)}$.
Before stating their result, we introduce
notation on cohomology.

\subsection{The Belkale-Kumar cohomology}
\label{sec:BKcohom}

Let $W$ denote the
Weyl group $G$ and let $s_\alpha\in W$ denote the simple reflexion
associated to $\alpha\in\Delta$. The simple reflections $s_\alpha$
generated $W$ and determine a length function
$l$.

Let $P$ be a standard (that is containing $B$) parabolic subgroup of
$G$
and let $W_P$ denote its Weyl group. 
The set of minimal length
representative of $W/W_P$ is denoted by $W^P$. 
For any $w\in W^P$, let $X_w=\overline{BwP/P}\subset G/P$
denote the Schubert variety. 
The Poincaré dual class $\sigma_w\in H^{2(\dim(G/P)-l(w))}(G/P,\ZZ)$ 
of the homology class of $X_w$ is a Schubert class. 
Let $\sigma_w^\vee$ be the Poincar\'e dual class of $\sigma_w$.

Recall that $H^*(G/P,\ZZ)=\oplus_{w\in W^P}\ZZ\sigma_w$.
We define the structure constants
$c({w_1},{w_2},{w_3})$ associated to three
Schubert classes $\sigma_{w_1}$, $\sigma_{w_2}$ and $\sigma_{w_3}$ by
the identity
$$
\sigma_{w_1} \sigma_{w_2}=\sum
c({w_1},{w_2},{w_3})
\sigma_{w_3}^\vee,
$$
where the sum runs over $w_3\in W^P$.
The cohomology ring of $G/P$ is graded by $\deg(\sigma_w)=2(\dim(G/P)-l(w))$ for
any $w\in W^P$.
In particular, $c(w_1,w_2,w_3)\neq 0$ implies
that
\begin{eqnarray}
  \label{eq:deg1}
  l(w_1)+l(w_2)+l(w_3) =2\dim(G/P).
\end{eqnarray}
Let $\Phi^+$ denote the set of positive roots.  Let $R^u(P)$ denote
the unipotent radical of $P$ and let $\Phi(G/P)$ denote  the
set of roots of $R^u(P)$.
For $w\in W$, let $\Phi(w)=\{\alpha\in\Phi^+\,:\,-w\alpha\in\Phi^+\} $
be the inversion set. Recall that $w\in W^P$ if and only if
$\Phi(w)\subset\Phi(G/P)$.
Condition~\eqref{eq:deg1} can be rewritten like
\begin{eqnarray}
  \label{eq:deg2}
  \sharp \Phi(w_1)+\sharp\Phi(w_2)+\sharp \Phi(w_3)
=2\sharp \Phi(G/P).
\end{eqnarray}

Let $L$ be the
Levi subgroup of $P$ containing $T$ and let $Z$ be the neutral component of
the center of $L$.
For any character $\chi$ of $Z$ we set 
\begin{eqnarray}
  \label{eq:deg3}
\Phi(G/P,\chi)=\{\alpha\in \Phi(G/P)\,:\,\alpha_{|Z}=\chi\}.
  \end{eqnarray}

For $w\in W^P$ we also set $\Phi(w,\chi)=\Phi(w)\cap\Phi(G/P,\chi)$.
Now condition~\eqref{eq:deg3} is equivalent to
 \begin{eqnarray}
  \label{eq:deg4}
  \sum_{\chi\in X^*(Z)}\bigg (
  \sharp\Phi(w_1,\chi)+\sharp\Phi(w_2,\chi)+\sharp \Phi(w_3,\chi)
\bigg )
=2 \sum_{\chi\in X^*(Z)}\sharp \Phi(G/P,\chi).
\end{eqnarray}

The main theorem of \cite{BK} combined with \cite[Proposition~]{RR}
allow to obtain the following result.

\begin{theo}
 \label{th:BK}
 Let $(\tau_1,\tau_2,\tau_3)\in (X_*(T)^+_\RR)^3$. 
Then $(\tau_1,\tau_2,\tau_3)\in \poly_{\Lie(K)}$ if and only if 
$$
\langle w_1\varpi_\beta,\tau_1\rangle+
\langle w_2\varpi_\beta,\tau_2\rangle+
\langle w_3\varpi_\beta,\tau_3\rangle\leq 0,
$$
for any simple root $\beta$,
any nonnegative integer $d$ and any $(w_1,w_2,w_3)$ such that 
\begin{eqnarray}
  \label{eq:BKcond2}
c(w_1,w_2,w_3)=1,
\end{eqnarray}
and for any $\chi\in X^*(Z)$
\begin{eqnarray}
  \label{eq:BKfiltration}
  \sharp\Phi(w_1,\chi)+\sharp\Phi(w_2,\chi)+\sharp \Phi(w_3,\chi)
=2 \sharp \Phi(G/P,\chi).
\end{eqnarray}
\end{theo}

\subsection{The multiplicative Horn problem}

\noindent{\bf The multiplicative Horn question.}
Let $\Orb_1$ and $\Orb_2$ be two conjugacy classes in $K$. Then the
product $\Orb_1\cdot\Orb_2=\{k_1k_2\,:\, k_1\in\Orb_1$ and $k_2\in
\Orb_2\}$ is stable by conjugacy. 
This article is concerned by the multiplicative Horn question:

\begin{center}
  What conjugacy classes are contained in $\Orb_1\cdot\Orb_2$?
\end{center}

\noindent{\bf Parametrization of the conjugacy classes.}
Let $\theta$ be the longest root of $G$ relatively to $T\subset B$.
The fundamental alcove in $X_*(T)_\RR$ is 
$$
\alc_*=\{\tau\in X_*(T)_\RR\,:\,\left\{
  \begin{array}{l}
    \langle\tau,\alpha\rangle\geq 0\quad\forall\alpha\in\Delta\\
\langle\tau,\theta\rangle\leq 1
  \end{array}
\right .
\}.
$$
Consider the exponential map 
$$
\begin{array}{cccl}
  \exp\::&\Lie(T_K)&\longto&T_K\\
&\mu&\longmapsto&\exp(\mu).
\end{array}
$$
Any conjugacy class in $ K$ contains a unique element of the form
$\exp(\sqrt{-1}\tau)$ for some $\tau\in \alc_*$ (see e.g.\ \cite[Chapter
IX. \S 5]{Bourb7-9}); we denote by
$\Orb_\tau$ the conjugacy class containing $\tau$. \\

\noindent{\bf The multiplicative Horn polytope.}
Answering the multiplicative Horn question is equivalent of describing
the set
$$
\poly_K=\{(\tau_1,\tau_2,\tau_3)\in \alc_*^3\,:\, \Orb_{\tau_1}\cdot
\Orb_{\tau_2}\cdot \Orb_{\tau_3}\ni e\},
$$
where $e$ is the unit element of $K$.
According to the convexity theorem proved by Meinrenken-Woodward
\cite{MW}, $\Delta$ is a convex polytope of nonempty interior 
in $\alc$.
Teleman-Woodward \cite{TW:parab} obtained an explicit list of
inequalities that characterize $\poly_K$.
The aim of this article is to determine a smaler list of inequalities
that still characterize the polytope.
Before stating Teleman-Woodward's theorem, we introduce
notation on quantum cohomology.

\subsection{Quantum cohomology of $G/P$}

Let $\Delta_P$ be the set of simple roots of $(L,T)$.
The Picard group $\Pic(G/P)$ identifies
with
$H^2(G/P,\ZZ)=\oplus_{\alpha\in\Delta-\Delta_P}\ZZ\sigma_{s_\alpha}$.
 We denote by $(\sigma_{s_\beta}^*)_{\beta\in\Delta-\Delta_P}$
 the $\ZZ$-basis of $\Hom( H^2(G/P,\ZZ),\ZZ)$ dual of the
 $(\sigma_{s_\beta})_{\beta\in \Delta-\Delta_P}$.

Let $\gamma\,:\,\PP^1\longto G/P$ be a curve. Identifying
the group $\Pic(\PP^1)$ to $\ZZ$ (by mapping ample line bundles on positive integers), the pullback of line bundles induces an
element of  $\Hom( H^2(G/P,\ZZ),\ZZ)$ called the degree of $\gamma$
and denoted by $\ud(\gamma)$.
 By construction $\ud(\gamma)\in\sum_{\beta\in\Delta-\Delta_P}\ZZ_{\geq 0}\sigma_{s_\beta}^*$.

Let $\rho$ and $\rho^L$ denote the half sum of positive roots of $G$
and $L$ respectively. For any $\beta\in\Delta-\Delta_P$, set 
\begin{eqnarray}
  \label{eq:defnbeta}
n_\beta=\langle\beta^\vee,2(\rho-\rho^L)\rangle,
\end{eqnarray}
where $\beta^\vee$ is the simple coroot. 
Fix $\ud=\sum_{\beta\in\Delta-\Delta_P}d_\beta\sigma_{s_\beta}^*$
for some $d_\beta\in \ZZ_{\geq 0}$.
Let $\Md$ be the moduli space of stable maps of degree $\ud$ with 3
marked points into $G/P$. It is a projective variety of dimension
$$
\dim(\Md)=\dim(G/P)+\sum_{\beta\in\Delta-\Delta_P} d_\beta n_\beta.
$$
It comes equipped with 3 evaluation maps $\ev_i\,:\,\Md\longto G/P$.
The Gromov-Witten invariant associated to three Schubert classes
(corresponding to $w_i\in W^P$) and a
degree $\ud$ is then the intersection number
$$
GW({w_1},{w_2},{w_3};\ud)=\int_\Md
ev_1^*(\sigma_{w_1})\cdot
ev_2^*(\sigma_{w_2})\cdot
ev_3^*(\sigma_{w_3}).
$$

For any $\alpha\in\Delta-\Delta_P$, we introduce a  variable
$q_\alpha$.
Consider the group
$$
\begin{array}{r@{\,}l}
QH^*(G/P,\ZZ):=&H^*(G/P,\ZZ)\otimes\ZZ[q_\beta\ \,:\,\beta
\in\Delta-\Delta_P]\\=&
\bigoplus_{w\in W^P}\ZZ[q_\beta\ \,:\,\beta
\in\Delta-\Delta_P]\sigma_w.
\end{array}
$$
The $\ZZ[q_\beta\ \,:\,\beta
\in\Delta-\Delta_P] $-linear quantum product $\star$ on $QH^*(G/P,\ZZ)$
is defined by, for any $w_1,w_2\in  W^P$, 
$$
\sigma_{w_1}\star \sigma_{w_2}=\sum
GW({w_1},{w_2},{w_3};\ud)
q^\ud\sigma_{w_3}^\vee,
$$
where the sum runs over $w_3\in W^P$ and over 
$\ud\in\sum_{\beta\in\Delta-\Delta_P}\ZZ_{\geq 0}\sigma_{s_\beta}^*$.
Here, if $\ud=\sum_{\beta\in\Delta-\Delta_P}d_\beta\sigma_{s_\beta}^*$
then $q^\ud=\prod_{\beta\in\Delta-\Delta_P}q_\beta^{d_\beta}$.

\subsection{Teleman-Woodward inequalities}

Fix for a moment a simple root $\beta$, the corresponding maximal
standard parabolic subgroup $P_\beta$ and the fundamental weight 
$\varpi_\beta$. 
Let $w_1$, $w_2$, and $w_3$ in $W^{P_\beta}$.
A degree for curves in $G/P_\beta$ is a nonnegative integer $d$. 
Consider the following linear inequality on points
$(\tau_1,\tau_2,\tau_3)$ in $X_*(T)\otimes \RR$:
$$
\iI_\beta(w_1,w_2,w_3;d)\qquad \langle w_1\varpi_\beta,\tau_1\rangle+
\langle w_2\varpi_\beta,\tau_2\rangle+
\langle w_3\varpi_\beta,\tau_3\rangle\leq d.
$$

We can now state Teleman-Woodward's theorem (see \cite{TW:parab}).

\begin{theo}[Teleman-Woodward (see \cite{TW:parab})]
 \label{th:TW}
 Let $(\tau_1,\tau_2,\tau_3)\in\alc_*^3$. 
Then $(\tau_1,\tau_2,\tau_3)\in \poly_K$ if and only if inequality
$\iI_\beta(w_1,w_2,w_3;d)$ is fulfilled for any simple root $\beta$,
any nonnegative integer $d$ and any $(w_1,w_2,w_3)$ such that 
\begin{eqnarray}
  \label{eq:TWcond}
GW(w_1,w_2,w_3;d\sigma_{s_\beta}^*)=1
\end{eqnarray}
in $G/P_\beta$.
\end{theo}
  
\subsection{Our main result}

Our main result is a raffinement of the
condition~\eqref{eq:TWcond}.

Here, $P$ is any standard parabolic subgroup of $G$. 
The grading on $H^*(G/P,\ZZ)$ extends to the quantum setting by setting
$\deg(q_\beta)=2n_\beta$ for any $\beta\in\Delta-\Delta_P$.
In particular, for $\ud=\sum_{\beta\in\Delta-\Delta_P}d_\beta\sigma_{s_\beta}^*$,
$GW(w_1,w_2,w_3;\ud)\neq 0$ implies
that
\begin{eqnarray}
  \label{eq:qdeg1}
  l(w_1)+l(w_2)+l(w_3) +\sum_{\beta\in\Delta-\Delta_P}d_\beta n_\beta=2\dim(G/P).
\end{eqnarray}
Condition~\eqref{eq:qdeg1} can be rewritten like
\begin{eqnarray}
  \label{eq:qdeg2}
  \sharp \Phi(w_1)+\sharp\Phi(w_2)+\sharp \Phi(w_3)+
\sum_{\beta\in\Delta-\Delta_P}d_\beta n_\beta
=2\sharp \Phi(G/P).
\end{eqnarray}
Set $h=\sum_{\beta\in\Delta-\Delta_P}d_\beta\beta^\vee$.
Since $2(\rho-\rho_L)=\sum_{\alpha\in\Phi(G/P)}\alpha$, condition~\eqref{eq:qdeg2} can be rewritten like
\begin{eqnarray}
  \label{eq:qdeg3}
  \sharp\Phi(w_1)+\sharp\Phi(w_2)+\sharp \Phi(w_3)+
\sum_{\alpha\in \Phi(G/P)}\langle h,\alpha\rangle
=2\sharp \Phi(G/P),
\end{eqnarray}
or like
 \begin{eqnarray}
  \label{eq:qdeg4}
  \sum_{\chi\in X^*(Z)}\bigg (
\sum_{i=1}^3 \sharp\Phi(w_i,\chi)
+
\sum_{\alpha\in \Phi(G/P,\chi)}\langle h,\alpha\rangle\bigg )
=2 \sum_{\chi\in X^*(Z)}\sharp \Phi(G/P,\chi).
\end{eqnarray}

\begin{theo}
 \label{th:main}
 Let $(\tau_1,\tau_2,\tau_3)\in\alc_*^3$. 
Then $(\tau_1,\tau_2,\tau_3)\in\Delta(K)$ if and only if inequality
$\iI_\beta(w_1,w_2,w_3;d)$ is fulfilled for any simple root $\beta$,
any nonnegative integer $d$ and any $(w_1,w_2,w_3)$ such that, in
$QH^*(G/P_\beta)$, 
\begin{eqnarray}
  \label{eq:TWcond2}
GW(w_1,w_2,w_3;d\sigma_{s_\beta}^*)=1,
\end{eqnarray}
and for any $\chi\in X^*(Z)$
\begin{eqnarray}
  \label{eq:filtration}
  \sharp\Phi(w_1,\chi)+\sharp\Phi(w_2,\chi)+\sharp \Phi(w_3,\chi)+
\sum_{\alpha\in \Phi(G/P,\chi)}d\langle \beta^\vee,\alpha\rangle
=2 \sharp \Phi(G/P,\chi).
\end{eqnarray}
\end{theo}

\subsection{Comparaison with Teleman-Woodward theorem}

We made some explicit computations using Anders Buch's  qcalc Maple
package, SageMath and Normaliz. The used programs, some files
containing explicit list of inequalities and additional
computation are available on author's webpage (see~\cite{MaPage}).
  
Here, we give some quantitative aspects for the group $G_2$ and the
groups of type $B$, $C$ or $D$ up to rank 6. 
More precisely, in the two last column of
Table~\ref{fig:comp}  appear the numbers of vertices and facets of
the polytope $\poly_K$. 
In  column ``MAX'', all the inequalities corresponding to nonzero
GW-invariants are counted.  In  column ``TW'', the number of
inequalities given by Theorem~\ref{th:TW} is given.
The inequalities obtained by combining Theorems~\ref{th:BK} and
\ref{th:TW} are counted.
In column ``Th~\ref{th:main}'', only the inequalities given by
Theorem~\ref{th:main} are counted. 
 All these numbers of inequalities include the $3*($rank$+1)$ inequalities of
dominancy and alcove. 
One can observe that Theorem~\ref{th:main} gives a list of
inequalities significantly smaller than the combination of
Theorems~\ref{th:BK}
and \ref{th:TW}.

\begin{figure}
\label{fig:comp}
  \centering
 
 $
\begin{array}{|r||r|r|r|r||r|r|}
\hline
{\rm Group}&\centering {\rm  MAX}
&\centering {\rm TW}& {\rm TWBK}&{\rm  Th~\ref{th:main}} & {\rm  Vertices} & {\rm
  Facets}\\
\hline
G_2&103 &82 &79 &48 &30&48\\
\hline
Sp(4)&43& 42& 41& 38&13&38\\
\hline
Sp(6)&363&329&296&200& 66&200\\
\hline
Sp(8)&4\,679& 3\, 604& 3\, 130&1\, 204&444&1\, 204\\
\hline
Sp(10)&75\, 665& 44\, 211& 38\, 795& 7\, 310&3\, 162&7\, 310\\
\hline  
Sp(12)&1\,422\,545& 556\,383& 500\,130& 43\,136&20\,839 &43\,136\\
\hline
Spin(7)&378& 322& 289&191&65&191\\
\hline
Spin(8)&1\, 434&1\, 347& 1\,164& 771&137&771\\
\hline
Spin(9)&4\, 940& 3\, 231& 2\, 748& 1\, 046&385& 1\, 046\\
\hline
Spin(10)&35\, 590&27\, 814& 23\, 050& 6\, 538&1\, 296&6\, 538\\
\hline
Spin(11)&79\, 813& 34\, 152& 28\, 636& 5\, 734&2\, 236&5\, 734\\
\hline
Spin(12)&889\,751& 485\,229&407\,856& 47\,141&?&?\\
\hline
Spin(13)&1\,499\,669&356\,942& 300\,776& 30\,753&12\,269 
&30\, 753\\
\hline 
\end{array}
$
 \caption{Explicit computations}
\end{figure}

It is worthy to observe that in any computed examples the number of facets
is equal to the number of inequalities given by our main result. 
One can conjecture that the list of inequalities given by
Theorem~\ref{th:main} is irredundant. The analogous result for the
additive Horn problem is proved in \cite{GITEigen}. 


\section{Notation}

In this section, we reintroduce more carefully and complete the
notation used in the introduction.

\subsection{Notation on the group $G$}

Let $G$ be a simple simply connected Lie group and $Z(G)$ its center. 
Set $G_{ad}=G/Z(G)$ and $T_{ad}=T/Z(G)$.
We fix a Borel subgroup $B$ of $G$ and a maximal torus $T$ contained
in $B$.
Let $\Phi$ and $\Phi^+$ denote the sets of roots and positive roots
respectively.
It $\alpha$ belongs to $\Phi$, $\alpha^\vee$ denote the corresponding
coroot.
The set of simple roots is denoted by $\Delta$.
For $\alpha\in \Delta$, $\varpi_\alpha\in X^*(T)$ denotes the
corresponding fundamental weight and $\varpi_{\alpha^\vee}\in X_*(T_{ad})$ denotes the
associated fundamental coweight.
Let $\rho$ be the half sum of the positive roots. Recall that
$\rho=\sum_{\alpha\in\Delta}\varpi_\alpha$.

Note that $X^*(T)=\oplus_{\alpha\in\Delta}\ZZ\varpi_\alpha$.
 Let $Q=\oplus_{\alpha\in\Delta}\ZZ\alpha\in X^*(T)$ denote the
 root latice. 
Similarly 
$X_*(T)=\oplus_{\alpha\in\Delta}\ZZ\alpha^\vee$ and 
$P^\vee:=\oplus_{\alpha\in\Delta}\ZZ\varpi_{\alpha^\vee}$.
Let $W$ be the Weyl group and $w_0$ be its longest element.
Let $h\in X_*(T)$. Write
$h=\sum_{\alpha\in\Delta}n_\alpha\alpha^\vee$. Note that
$\langle\rho,h\rangle= \sum_{\alpha\in\Delta}n_\alpha$.
The dominance order on $X_*(T)$ is the partial order $\geq$ defined by 
$$
h\geq h'\ \iff\ h-h'\in\sum_{\alpha\in\delta}\ZZ_{\geq 0}\alpha^\vee.
$$

Let $\theta$ denote the longest root.
Set $X^*(T)_\RR=X^*(T)\otimes \RR$ and its dual space $X_*(T)_\RR=X_*(T)\otimes \RR$.
 There exists a $W$-invariant Euclidean scalar
product $(\ ,\ )$ on $X^*(T)_\RR$. Moreover, it is unique modulo
positive scalar. We fix a choice by assuming that
$(\theta ,\theta )=2$.
Using $(\ ,\ )$, we identify $X^*(T)_\RR$ with $X_*(T)_\RR$.
The transition relations are, for any $\alpha\in \Delta$
$$
\alpha^\vee=(\frac 2 {(\alpha,\alpha)}\alpha,\square ),\quad
\varpi_{\alpha^\vee}=(\frac 2 {(\alpha,\alpha)}\varpi_\alpha,\square ).
$$
Observe that $\frac 2 {(\alpha,\alpha)}=1, 2$ or $3$, depending if
$\alpha$ is not short, short in type $\neq G_2$ and short in type
$G_2$.
In particular 
$X^*(T)\supset Q\supset X_*(T)\subset P^\vee\subset X^*(T)$.

A one parameter subgroup $\tau$ of $T$ is said to be dominant if
$\langle\tau,\alpha\rangle\geq 0$ for any $\alpha\in\Delta$. The set
of dominant one parameter subgroups is denoted by
$X_*(T)^+$. Similarly; $\lambda\in X^*(T)$ is dominant, or belongs to
$X^*(T)^+$ if $\langle\lambda,\alpha^\vee\rangle\geq 0$. We extend
these definitions and notations to $X^*(T)_\RR$ and $X_*(T)_\RR$.

For $\tau\in X_*(T)$, we denote by $P(\tau)$ the set $g\in G$ such
that $\tau(t)g\tau(t^{-1})$ has a limit in $G$ when $t$ goes to
$0$. It is a parabolic subgroup of $G$. It contains $B$ if and only if
$\tau$ is dominant.

\subsection{The affine Kac-Moody Lie algebra}

Endow $\hat \kmL\lg=\lg\otimes \CC((z))\oplus\CC c\oplus
\CC d$ with the usual Lie bracket (see \eg \cite[Chap XIII]{Kumar:KacMoody}).
Set $\hat\lh=\Lie(T)\oplus\CC c\oplus\CC d$.
We identify $\Lie(T)^*$ with the orthogonal of $\CC c\oplus\CC d$ in
$\hat\lh$.
Define $\Lambda$ and $\delta$ in $\hat\lh^*$ by
$$
\begin{array}{l}
\delta\,:\, \lh\longmapsto 0, c\longmapsto 0,
d \longmapsto 1;\\
  \Lambda\,:\, \lh\longmapsto 0, c\longmapsto 1,
d \longmapsto 0.
\end{array}
$$
The simple roots of $\hat \kmL\lg$ are 
$$
\alpha_0=\delta-\theta,\alpha_1,\dots,\alpha_l.
$$
For any fundamental weight $\varpi$ of $\lg$, set
$\hat\varpi=\varpi+\varpi(\theta^\vee)\Lambda\in\hat\lh^*$.
Fix a numbering $\alpha_1,\dots,\alpha_l$ of the simple roots of
$\lg$. Set $\hat\varpi_0=\Lambda$.
The fundamental weights of $\hat \kmL\lg$ are
$\hat\varpi_0,\,\hat\varpi_1,\dots,\hat\varpi_l$.
Set
$$
\hat\lh_\ZZ^*=\ZZ \hat\varpi_0\oplus\cdots\oplus
\ZZ \hat\varpi_l\oplus\ZZ\delta,
$$
and
$$
\hat\lh_\ZZ^{*+}=\ZZ_{\geq 0}\varpi_0\oplus\cdots\oplus
\ZZ_{\geq 0}\varpi_l\oplus\ZZ\delta.
$$
Fix $\hat \lambda=\lambda+\bl\Lambda+z\delta\in \hat\lh_\ZZ^*$ with
$\lambda\in\lh^*$, $\bl\in\ZZ_{\geq 0}$ and $z\in\ZZ$.
If $\hat \lambda\in \hat\lh_\ZZ^{*+}$, that is, if $\langle
\lambda,\theta^\vee\rangle\leq \bl$,
 then there exists a simple $\hat \kmL\lg$-module $\kmV(\hlambda)$ of
highest weight $\hlambda$.
 The subspace of $\kmV(\lambda+\bl \Lambda)$ annihilated by
 $\lg\otimes z\CC[z]$ is isomorphic as a $\lg$-module to
 $V(\lambda)$. 

\label{sec:Waff}
Let $s_0,s_1,\dots,s_l$ be the set of simple reflections. They
generate the  affine Weyl group $\Waff$ which is isomorphic to $W\ltimes Q^\vee$. 
Moreover, $\Waff$  is a Coxeter group and the
 length is given by
$$
l(t^hw)=\sum_{
  \begin{array}{l}
    \alpha\in\Phi^+\\
w^{-1}\alpha\in \Phi^+
  \end{array}
}
|\langle h,\alpha\rangle|+
\sum_{
  \begin{array}{l}
    \alpha\in\Phi^+\\
w^{-1}\alpha\in \Phi^-
  \end{array}
}
|\langle h,\alpha\rangle-1|.
$$

The group $\Waff$ acts on $\hat\lh^*$. In particular the action of
$Q^\vee$ is given by
$$
h\longmapsto T_h\,:\,
\begin{array}{ccc}
  \hat\lh^*&\longto&\hat\lh^*\\
\chi&\longmapsto&\chi+\chi(c)(h,\square)
-[\chi(h)+\frac 1 2 ( h,h)
\chi(c)]\delta.
\end{array}
$$

\subsection{The fusion product}

If $\la$ is a Lie algebra and $M$ is a $\la$-module, we denote by
$
[M]_\la,
$
the biggest quotient of $M$ where $\la$ acts trivially.
 
Let $\lambda_1$, $\lambda_2$, and $\lambda_3$ be three dominant weights
of $\lg$ and $\bl\in\ZZ_{\geq 0}$ such that
$$
\langle\lambda_i,\theta^\vee\rangle\leq \bl,\quad{\rm for\ any\ } i=1,2,3.
$$
Then $\lambda_i+\bl\Lambda\in \hat\lh_\ZZ^{*+}$, and we can consider 
the $(\hat \kmL\lg)^3$-module $\kmV(\lambda_1+\bl\Lambda)\otimes
\kmV(\lambda_2+\bl\Lambda)\otimes \kmV(\lambda_3+\bl\Lambda)$.

Consider $\PP^1$ with three pairwise distinct marked points $p_1$, $p_2$ and
$p_3$.
Consider the ring of regular functions ${\mathcal
  O}(\PP^1-\{p_1,p_2,p_3\})$ and the Lie algebra $\lg\otimes {\mathcal
  O}(\PP^1-\{p_1,p_2,p_3\})$.
For any $p_i$, by fixing a local coordinate $z_i$ around this point of
$\PP^1$, one gets a morphism ${\mathcal
  O}(\PP^1-\{p_1,p_2,p_3\})\longto\CC((z))$. In particular, we just
defined three morphisms
$\lg\otimes {\mathcal
  O}(\PP^1-\{p_1,p_2,p_3\})\longto \lg\otimes\CC((z))$, or one
morphism
$\lg\otimes {\mathcal
  O}(\PP^1-\{p_1,p_2,p_3\})\longto (\lg\otimes\CC((z)))^3$.
This defines an action of $\lg\otimes {\mathcal
  O}(\PP^1-\{p_1,p_2,p_3\})$ on the $(\hat \kmL\lg)^3$-module $\kmV(\lambda_1+\bl\Lambda)\otimes
\kmV(\lambda_2+\bl\Lambda)\otimes \kmV(\lambda_3+\bl\Lambda)$.
The Vacua space is defined by
$$
V_{\PP^1}(\lambda_1,\lambda_2,\lambda_3,\bl)=\bigg(\kmV(\lambda_1+\bl\Lambda)\otimes
\kmV(\lambda_2+\bl\Lambda)\otimes \kmV(\lambda_3+\bl\Lambda) 
\bigg)_{\lg\otimes{\mathcal
  O}(\PP^1-\{p_1,p_2,p_3\})}
$$
It is proved to be finite dimensional (see e.g. \cite{Beauville:fusion}). Moreover, the fusion product
$\circledast_\text{\cursive l}$ is defined by
$$
V(\lambda_1)\circledast_\text{\cursive
  l}V(\lambda_2)=\sum_{\langle\lambda_3,\theta^\vee\rangle\leq
  \text{\cursive l}} \dim(V_{\PP^1}(\lambda_1,\lambda_2,\lambda_3;
\bl)) V(-w_0\lambda_3).
$$
The product $\circledast_\bl$ is associative and commutative (see
e.g. \cite{Beauville:fusion}). 

\subsection{The fusion product polytope}

The fundamental alcove in $X^*(T)_\QQ$ is 
$$
\alc^*_\QQ=\{\lambda\in X^*(T)_\QQ\,:\,\left\{
  \begin{array}{l}
    \langle\lambda,\alpha^\vee\rangle\geq 0\quad\forall\alpha\in\Delta\\
\langle\lambda,\theta^\vee\rangle\leq 1
  \end{array}
\right .
\}.
$$
For any $\lambda\in\alc^*_\QQ$, $\bl\in\ZZ_{>0}$ if
$\bl\lambda+\bl\Lambda\in\hat\lh_\ZZ^*$ then it is
dominant. Set
$$
\poly_\circledast=\{(\lambda_1,\lambda_2,\lambda_3)\in(\alc^*_\QQ)^3\;:\;
\exists \text{\cursive l}>0\quad
 V(-\text{\cursive l} w_0\lambda_3)\subset V({\text{\cursive l}\lambda_1})\circledast_\text{\cursive l} V(\text{\cursive l}\lambda_2)\}.
$$ 

\begin{theo}
 \label{th:mainFusion}
 Let $(\lambda_1,\lambda_2,\lambda_3)\in(\alc^*_\QQ)^3$. 
Then $(\lambda_1,\lambda_2,\lambda_3)\in \poly_\circledast$ if and only
if
\begin{eqnarray}
   \langle w_1\varpi_{\beta^\vee},\lambda_1\rangle+
\langle w_2\varpi_{\beta^\vee},\lambda_2\rangle+
\langle w_3\varpi_{\beta^\vee},\lambda_3\rangle\leq\frac 2 {(\beta,\beta)} d.
\end{eqnarray}
 for any simple root $\beta$,
any nonnegative integer $d$ and any $(w_1,w_2,w_3)\in (W^{P_\beta})^3$ such that 
\begin{eqnarray}
GW(w_1,w_2,w_3;d\sigma_{s_\beta}^*)=1,
\end{eqnarray}
and for any $\chi\in X^*(T)$
\begin{eqnarray}
  \sharp\Phi(w_1,\chi)+\sharp\Phi(w_2,\chi)+\sharp \Phi(w_3,\chi)+
\sum_{\alpha\in \Phi(G/P,\chi)}\langle h,\alpha\rangle
=2 \sharp \Phi(G/P_\beta,\chi).
\end{eqnarray}
\end{theo}

Let $\lambda\in X^*(T)_\QQ$. 
Since $\theta^\vee=(\theta,\square)$, $\lambda$ belongs to
$\alc_\QQ^*$ if and only if $\lambda\in\alc_*$. 
Theorems~\ref{th:main} and \ref{th:mainFusion} are equivalent knowing
the following. 

\begin{theo}[see \cite{TW:parab}]
 \label{th:eigen=fusion}
 Let $(\lambda_1,\lambda_2,\lambda_3)\in(\alc^*_\QQ)^3$. 
Then $(\lambda_1,\lambda_2,\lambda_3)\in \poly_\circledast$ if and only
if $((\lambda_1,\square),(\lambda_2,\square)),(\lambda_3,\square))\in \poly_K$. 
\end{theo}


\section{The affine Grassmannian}

In this section we collect some results and notation on the affine
Grassmannian $\Gra$ of $G$.
Set $
  L_{alg}G=G(\CC[z,z^{-1}])$ and 
$L^{>0}_{alg}G=G(\CC[z])$.
Consider the affine grassmannian $\Gra= L_{alg}G/ L_{alg}^{>0}G$.

\subsection{Line bundles}

Let $\bar L_{alg}G=\CC^*\ltimes L_{alg}G$ and $\kmG$ denote the
affine Kac-Moody group associated to $G$; it is a  
central extension of $\bar L_{alg}G=\CC^*\ltimes L_{alg}G$. 
The maximal torus of $\kmG$ containing $T$ is denoted by $\hat T$; its
Lie algebra is $\hat\lh$ and its character group is $\hat\lh^*_\ZZ$.
The group $\kmG$ acts on $\Gra$.

Let $\bl\in\ZZ$. There exists a unique $\kmG$-linearized line bundle
$\Li(\bl\Lambda)$ on $\Gra$ such that $\hat\lh$ acts on the fiber over
the base point of $\Gra$ by the weight $-\bl\Lambda$ (see \eg
\cite[Chap VII]{Kumar:KacMoody}).
Moreover, $H^0(\Gra,\Li(\bl\Lambda))$ is zero if $\bl<0$ and
isomorphic to the dual of $\kmV(\bl\Lambda)$ if $\bl\geq 0$.\\

Recall that $\kmG$ is a central extension of the semidirect product
$\CC^*\ltimes L_{alg}G$:

\begin{center}
\begin{tikzpicture}
  \matrix (m) [matrix of math nodes,row sep=3em,column sep=4em,minimum width=2em]
  {1&\CC^*&\kmG&\CC^*\ltimes L_{alg}G&1.\\};
  \path[-stealth]  
    (m-1-1) edge  (m-1-2) 
    (m-1-2) edge (m-1-3) 
    (m-1-3) edge (m-1-4)
    (m-1-4) edge (m-1-5);       
\end{tikzpicture} 
\end{center}

This exact sequence splits canonically  over $L_{alg}^{>0}G$. In
particular, $\Li(\bl\Lambda)$ admits a $L_{alg}^{>0}G$-linearization.

\subsection{The Cartan decomposition}

Any one parameter subgroup $h$ of $T$ can be seen as an element of
$L_{alg}G$. Its image in $
\Gra$ is denoted by $L_h$.
Then $\{L_h\,:\,h\in X_*(T)\}$ is the set of $T$-fixed points in
$\Gra$.
The $L^{>0}_{alg}G$-orbit of $L_h$ only depends on the $W$-orbit of $h$
in $X_*(T)$; it is denoted by $\Gra_h$. It is a quasiprojective
variety of  finite  dimension $\langle\rho,h\rangle$ (if $h\in
X_*(T)^+$) and 
the Cartan decomposition asserts that 
$$
\Gra=\bccup_{
    h\in X_*(T)^+
}\Gra_h.
$$
The closure of $\Gra_h$ is described by the order $\leq$:
$$
\overline{\Gra_h}=\bccup_{
\begin{array}{l}
    h'\in X_*(T)^+
\\
h'\leq h
  \end{array}
} \Gra_{h'}.
$$

There exists a unique one parameter subgroup $\delta^\vee$ of $\hat T$
such that $\langle\delta^\vee,\delta\rangle=1$,  $\langle\delta^\vee,\Lambda\rangle=0$, 
and  $\langle\delta^\vee,\delta\rangle=0$, for any $\chi\in X^*(T)$.
The irreducible components of
$\Gra^{\delta^\vee}$ are the $G.L_h$ for $h\in X_*(T)$.
Moreover, $G.L_h$ is isomorphic to $G/P(h)$.
Then
$$
\Gra_h=\{x\in \Gra\;:\;\lim_{t\to 0}\delta^\vee(t)x\in G.L_h\}.
$$

\noindent{\bf Raffinement.}
Consider the evaluation morphism $ev_0\,:\,L_{alg}^{>0}G\longto G$ at
$z=0$. Set $\Iw=ev_0^{-1}(B)$.
We have 
$$
\Gra=\bccup_{
    h\in X_*(T)
}\Iw L_h.
$$


\subsection{The Birkhoff decomposition}

Consider the action of the group $L^{<0}_{alg}G=G(\CC[z^{1}])$ on
$\Gra$. Its orbits are parametrized by $X_*(T)^+$ and setting
$\Gra^h=L^{<0}_{alg}G L_h$, we have
$$
\Gra=\bccup_{h\in X_*(T)^+}\Gra^h.
$$
Moreover
$$
\overline{\Gra^h}=\bccup_{
\begin{array}{l}
    h\in X_*(T)^+
\\
h\leq h'
  \end{array}}
\Gra^{h'}.
$$

For any $h\in X_*(T)^+$,
the orbit $\Gra_h$ has codimension $\langle\rho,h\rangle$.
Moreover
$$
\Gra^h=\{x\in \Gra\;:\;\lim_{t\to \infty}\delta^\vee(t)x\in G.L_h\}.
$$

\noindent{\bf Raffinement.}
Let $ev_\infty\,:\,L_{alg}G\longto G$ denote the 
evaluation at $z^{-1}=0$.
 Set $\Iw^-=ev_\infty^{-1}(B^-)$.
Then
$$
\Gra=\bccup_{
    h\in X_*(T)
}\Iw^- L_h.
$$
For any $h\in X_*(T)^+$,
\begin{eqnarray}
  \label{eq:GrahBmoins}
\Gra^h=\bigcup_{w\in W}\Iw^-L_{wh}.
\end{eqnarray}
Consider $\rho^\vee$ the half sum of the positive coroots. It is a
dominant and regular one parameter subgroup of $T_{ad}$.
Moreover, 
\begin{eqnarray}
  \label{eq:BmoinsBB}
\Iw^-L_{h}=\{\monx\in\mvar\,|\,\lim_{t\to\infty}(\delta^\vee+\rho^\vee)(t).\monx=L_h\}.
\end{eqnarray}

\subsection{The Peterson decomposition}

Consider the group $L_{alg}U$. 

\begin{theo}\label{th:decPeterson}
We have
$$
L_{alg}G=\bigcup_{w\in\Waff} \Iw w L_{alg}U =
\bigcup_{w\in\Waff} \Iw^- w L_{alg}U.
$$
\end{theo}

For later use, we prove the following lemma due to Peterson
\cite{Pet:course}. It implies easily Theorem~\ref{th:decPeterson}.

\begin{lemma}
\label{lem:degphi}
Let $g\in L_{alg}G$. Assume that
$g\in \Iw^-w_1z^{h_1}L_{alg}U$ and  $g\in \Iw w_2z^{h_2}L_{alg}U$ for $w_1$, $w_2$ in $W$,  and $h_1,\,h_2$ in $Q^\vee$.

Consider  $\phi\,:\,\PP^1\longto G/B$ that extends $\CC^*\longto G/B,
z\longmapsto g(z) B/B$.
\\

The map $\phi$ has degree
$h_2-h_1\in\Hom(Pic(G/B)=X(T),\ZZ)=\Hom(X(T),\ZZ)$. 
Moreover $\phi(\infty)\in B^-w_1B/B$ and $\phi(0)\in B w_2B/B$. 
\end{lemma}

\begin{proof}
  Write $g(z)=b(z) \tilde{w}_2 t^{h_2}u(z)$ with $b(z)\in \Iw$, $\tilde{w}_2\in N(T)$ a representant of $w_2$ and $u(z)\in L_{alg}U$.
Then $\phi(z)=b(z) \tilde{w}_2 t^{h_2}u(z).B/B=b(z) w_2.B/B$. 
Since $b(0)\in B$, we obtain $\phi(0)\in B w_2B/B$.
Similarly $\phi(\infty)\in B^-w_1B/B$.\\

It remains to compute the degree of $\phi$. Fix a dominant weight $\lambda$ of $G$. 
Consider  the irreducible $G$-representation $V(\lambda)$ of highest
weight $\lambda$ and an highest weight vector $v_\lambda$. 
Consider the  morphism $G/B\longto\PP(V(\lambda)),\,gB/B\longmapsto
g[v_\lambda]$ and its  composition  $\phi_\lambda\,:\,\PP^1\longto\PP(V(\lambda))$ with $\phi$. 
It remains to prove that  $deg(\phi_\lambda)=\langle\lambda,h_1-h_2\rangle$.

We reuse the writing $g(z)=b(z) \tilde{w}_2 t^{h_2}u(z)$:
$$
\forall z\in\CC^*\qquad \phi_\lambda(z)= [b(z) \tilde{w}_2 t^{h_2}u(z)\cdot v_\lambda]=
[z^{\langle \lambda, h_2\rangle} b(z) \tilde{w}_2\cdot v_\lambda].
$$
Since $b(z)$ is polynomial in $z$, this implies that the valuation (at
zero) of $z^{\langle
  \lambda, h_2\rangle} b(z) \tilde{w}_2\cdot v_\lambda$ is at least $\langle \lambda, h_2\rangle$.
Since $b(z)$ has a limit in $G$ at $z=0$, $b(z) \tilde{w}_2\cdot
v_\lambda$ has a nonzero  limit in $V(\lambda)$ at $z=0$. 
Hence the valuation of $z^{\langle \lambda, h_2\rangle} b(z)
\tilde{w}_2\cdot v_\lambda$ is  exactly $\langle \lambda, h_2\rangle$.

A similar computation with $\Iw^-w_1t^{h_1}L_{alg}U$ shows that
the degree of $z\longmapsto g(z)v_\lambda$ is exactly $\langle \lambda, h_1\rangle$.
Finally, the degree of $\phi_\lambda$ is $\langle \lambda, h_1-h_2\rangle$.
\end{proof}

\bigskip
For $h\in X_*(T)$, set $S_h=L_{alg}U L_h$. Then 
$$
\Gra=\bccup_{h\in X_*(T)}S_h,
$$
and
\begin{eqnarray}
  \label{eq:Shbar}
\overline{S_h}=\bccup_{
\begin{array}{l}
    h'\in X_*(T)
\\
h'\leq h
  \end{array}}
S_{h'}.
\end{eqnarray}
The orbit $S_h$ has neither finite dimension nor finite codimension.
The fixed points of $\rho^\vee$ are the $L_h$ for $h\in X_*(T)$ and
$$
S_h=\{x\in \Gra\;:\;\lim_{t\to 0}\rho^\vee(t)x=L_h\}.
$$

\bigskip{\bf Variation.} Let $P\supset B$ be a parabolic subgroup and
consider $L_{alg}P$. There exists a surjective group morphism
$$
\Chi\;:\;
L_{alg}P\longto \Hom(X^*(P),\ZZ)
$$
defined as follows. 
Let $p\in L_{alg}P$ considered as a regular map $p\,:\,\CC^*\longto
P$ and $\chi\in X^*(P)$. 
Then $\chi\circ p$ is a regular map from $\CC^*$ to $\CC^*$. 
Hence, there exist $n\in\ZZ$ and $\lambda\in\CC^*$ such that
$\chi(p)(z)=\lambda z^n$, for any $z\in\CC^*$. Then $\Chi(p)(\chi)$ is
defined to be   $n$.
The kernel of $\Chi$ is denoted by $(L_{alg}P)_0$.

Let $L$ be the Levi subgroup of $P$ containing $T$ and $L^{\rm ss}$ be
its semisimple part. 
Two orbits $S_h$ and $S_{h'}$ are contained in the same
$(L_{alg}P)_0$-orbit if and only if $h-h'\in X_*(T\cap L^{\rm  ss})$.
 Since $X_*(T)=\oplus_{\alpha\in\Delta}\ZZ\alpha^\vee$ and 
$X_*(T\cap L^{\rm  ss})=\oplus_{\alpha\in\Delta_P}\ZZ\alpha^\vee$, we
get 
\begin{eqnarray}
  \label{eq:43}
  \Gra=\bccup_{h\in \oplus_{\alpha\in\Delta-\Delta_P}\ZZ\alpha^\vee} S_h^P,
\end{eqnarray}
where $S_h^P=(L_{alg}P)_0 L_h$.
Let $\tau\in X_*(T)$ such that $P=P(\tau)$.
The irreducible components of the fixed point set $\Gra^\tau$ are the orbit
$C_h^L:=L_{alg}L^{\rm ss}.L_h$ for $h\in \oplus_{\alpha\in\Delta-\Delta_P}\ZZ\alpha^\vee$. 
 Moreover, 
 \begin{eqnarray}
   \label{eq:53}
   S_h^P=\{x\in \Gra\,:\,\lim_{t\to 0}\tau(t)x\in C_h^L\}.
 \end{eqnarray}

Observe that $C_h^L=L_{alg}L^{\rm ss}.L_h$ is well-defined for any
$h\in X_*(T)$, but depends only on the class of $h$ in
$X_*(T)/X_*(T\cap L^{\rm ss})$. Above, we choose
$\oplus_{\alpha\in\Delta-\Delta_P}\ZZ\alpha^\vee$ as a complete system of
representant for this quotient. The following result due to
Peterson-Woodward gives another representative (see~\cite[Lemma~1]{Woodward:compare}):

\begin{lemma}
\label{lem:PW}
  Each class $h\in X_*(T)/X_*(T\cap L^{\rm ss})$ has a unique
  representative $h_{PW}\in X_*(T)$ such that 
$$
\langle h_{PW},\alpha\rangle=0{\mbox \ or }-1,\,\mbox{for any }\alpha\in \Phi^+\cap\Phi(L).
$$
\end{lemma}
It is easy to check that the standard Iwahori subgroup of
$L_{alg}L^{\rm ss}$ fixes $L_{h_{PW}}$.

\section{GIT for $L_{alg}^{<0}G$ acting on the affine Grassmannian}

\subsection{Fusion product and $L_{alg}^{<0}G$-invariant sections}

We think about $z^{-1}$ as a coordinate on $\PP^1-\{0\}$.
Hence, for $p\in \PP^1-\{0\}$, we have a morphism of evaluation 
$ev_p\,:\,L^{<0}_{alg}G\longto G$, $g(z^{-1})\mapsto g(p)$.

Let $\mvar =\Gra\times (G/B)^3$. 

Recall that we have  fixed three pairwise distinct points $p_1,p_2$ and $p_3$ in
$\PP^1-\{0\}$.

Let $\bl\in\ZZ_{\geq 0}$ and $\lambda_i$ (for
$i=1,\dots,3$) be three dominant characters  of $B$. 
Let $\Li=\Lila$ be the associated line bundle on $\mvar $.\\

\begin{lemma}
\label{lem:vacuasection}
The dual of the Vacua space $V_{\PP^1}(\lambda_1,\lambda_2,\lambda_3,\bl)$ is
isomorphic to the  space of $L_{alg}^{<0}G$-invariant sections of $\Li$
\end{lemma}

\begin{proof}
By \cite[Corollary~2.5]{Beauville:fusion}, the Vacua space is
isomorphic to 
   $$
\bigg(\kmV(\bl\Lambda)\otimes
V(\lambda_1)\otimes V(\lambda_2)\otimes V(\lambda_3)
\bigg)_{\lg\otimes{\mathcal
  O}(\PP^1-\{0\})}.
$$
Its dual is the set of $\lg\otimes{\mathcal
  O}(\PP^1-\{0\})$-invariant vectors in 
$\bigg(\kmV(\bl\Lambda)\otimes
V(\lambda_1)\otimes V(\lambda_2)\otimes V(\lambda_3)
\bigg)^*$.
Since $V(\lambda_1)\otimes V(\lambda_2)\otimes V(\lambda_3)$ is finite
dimensional, $\bigg(\kmV(\bl\Lambda)\otimes
V(\lambda_1)\otimes V(\lambda_2\otimes V(\lambda_3)
\bigg)^*=\kmV(\bl\Lambda)^*\otimes
V(\lambda_1)^*\otimes V(\lambda_2)^*\otimes V(\lambda_3)^*$.
This space is isomorphic to $H^0(\mvar,\Li)$. The lemma follows.
\end{proof}

\subsection{Convex numerical function}
\label{sec:cvxe}

Let $E$ be a finite dimensional real vector space and let $E^*$ denote
its dual space.
Let $\mu\,:\,E^*\longto \RR$ be a function. It is said to be {\it positively
homogeneous} if $\mu(t\varphi)=t\mu(\varphi)$ for any $\varphi\in E^*$
and any nonnegative real number $t$.
The positively homogeneous function $\mu$ is said to be {\it convex}
if 
$$
\forall\varphi,\psi\in E^*\qquad\mu(\varphi+\psi)\geq \mu(\varphi)+\mu(\psi).
$$

\begin{NB}
Pay attention to our convention which is nonstandard in convex
analysis. Our convention is that  of toric geometry.
\end{NB}

To any positively homogeneous convex function $\mu$ is associated the
compact convex set 
$$
C_\mu=\{x\in E\,:\,\forall\varphi\in E^*\quad \varphi(x)\geq\mu(\varphi)\}.
$$
The correspondance $\mu\mapsto C_\mu$ is bijective since, by
Hahn-Banach's  theorem 
$$
\mu(\varphi)=\inf_{x\in C_\mu}\varphi(x).
$$
The function $\mu$ is said to be {\it piecewise linear} if there
exists a fan $\Sigma$ in $E^*$ such that the restrictions of $\mu$ to its maximal
cones are linear.
Observe that $\mu$ is piecewise linear if and only if $C_\mu$ is
polyhedral. In this case, to any maximal cone $\sigma$ in $\Sigma$ we
associate $x_\sigma$ which is the unique point in $C_\mu$ such that
$\varphi(x_\sigma)=\mu(\varphi)$ for any $\varphi\in\sigma$.
 Then, $C_\mu$ is the convex hull of the points $x_\sigma$. 
Dually, $C_\mu$ is the set of $x\in E$ such that
$\varphi(x)\geq\mu(\varphi)$ for any $\varphi$ on a ray of $\Sigma$.

The point $0$ belongs to $C_\mu$ if and only if $\mu(\varphi)\leq 0$
for any $\varphi\in E^*$.
Fix a scalar product $(\,,\,)$ on $E$ and hence on $E^*$. 
We denote by $\Vert\,\Vert$ the associated norm.
Then $0$ does not belong to $C_\mu$ if and only if
$\sup_{\Vert\varphi\Vert=1}\mu(\varphi)>0$. 
In this case, this sup is reached for a unique $\varphi_0\in E^*$ such
that    $\Vert\varphi_0\Vert=1$.
Consider the orthogonal projection $x_0\in C_\mu$ of $0$ on
$C_\mu$. Then
$\varphi_0=\frac 1{\Vert x_0\Vert} (x_0,\square)$ and
$\mu(\varphi_0)$ is the distance from $0$ to the convex $C_\mu$.
Moreover  $\varphi_0\in E^*$ is characterized
by the following properties:
\begin{enumerate}
\item $\Vert\varphi_0\Vert=1$;
\item $\mu(\varphi_0)x_0$ belongs de $C_\mu$, where $x_0$ is given by
  $(x_0,\square)=\varphi$.
\end{enumerate}

\subsection{Numerical semistability}

\label{sec:defxss}

Let $\monx\in \mvar $. Observe that the closure $\overline{T.x}$ of the orbit
$T.x$ is a finite dimensional projective variety. 
Let $\tau$ be a one parameter subgroup  of $T$.
Consider $\monz=\lim_{t\to 0}\tau(t)\monx$. Recall from \cite{GIT}
that 
$\mu^\Li(x,\tau)\in\ZZ$ is characterized by 
$\tau(t).\tilde\monz=t^{-\mu^\Li(x,\tau)}$ for any $t\in\CC^*$ and any
$\tilde\monz$ in the fiber $\Li_\monz$ over $\monz$ in $\Li$.
The map $\tau\longmapsto \mu^\Li(x,\tau)$ extends uniquely to a
continuous, positively homogeneous map from $X_*(T)_\RR$ to $\RR$.
This extension, still denoted by $\mu$, is convex. 
\\

\begin{defin}
  The point $\monx\in \mvar $ is said to be {\it numericaly semistable relatively to
    $\Li$} if for any $g\in L_{alg}^{<0}G$ and any dominant one parameter
  subgroup $\tau$ of $T$, we have
$$
\mu^\Li(gx,\tau)\leq 0.
$$
\end{defin}

Let $\mvar^{\rm nss}(\Li)$ denote the set of numericaly semistbale
points in $\mvar$. A point that is not semistable is said to be unstable.

Consider the set 
$
\cone^{\rm nss}(\mvar)$ of $(\lambda_1,\lambda_2,\lambda_3,\bl)$ in
$(X^*_\QQ(T))^3\times\QQ$
such that there exists $k>0$ satisfying
\begin{enumerate}
\item $k\lambda_1,k\lambda_2,k\lambda_3$ are dominant integral weights
  and $k\bl\in\ZZ_{> 0}$;
\item $\frac{\lambda_1}\bl$, $\frac{\lambda_2}\bl$ and $\frac{\lambda_3}\bl$ belong to the
  alcove $\alc^*$;
\item $\mvar^{\rm
    nss}(\Li(k\bl\Lambda)\otimes\Li(k\lambda_1)\otimes\Li(k\lambda_2)\otimes
  \Li(k\lambda_3))$ is not empty.
\end{enumerate}

Our main statement can be formulated in terms of numerical
semistability as follows. 

\begin{theo}
 \label{th:mainGIT}
 Let $(\lambda_1,\lambda_2,\lambda_3)\in(X^*(T)^+_\QQ)^3$ and $\bl\in
 \ZZ_{>0}$ such that $\frac{\lambda_1}\bl$, $\frac{\lambda_2}\bl$ and $\frac{\lambda_3}\bl$ belong to the
  alcove $\alc^*$. 
Then $(\lambda_1,\lambda_2,\lambda_3,\bl)\in \cone^{\rm nss}(\mvar)$  if and only
if
\begin{eqnarray}
  \label{eq:mainGIT}
   \langle w_1\varpi_{\beta^\vee},\lambda_1\rangle+
\langle w_2\varpi_{\beta^\vee},\lambda_2\rangle+
\langle w_3\varpi_{\beta^\vee},\lambda_3\rangle\leq\frac 2
{(\beta,\beta)} \bl d.
\end{eqnarray}
 for any simple root $\beta$,
any nonnegative integer $d$ and any $(w_1,w_2,w_3)\in (W^{P_\beta})^3$ such that 
\begin{eqnarray}
GW(w_1,w_2,w_3;d\sigma_{s_\beta}^*)=1,
\end{eqnarray}
and for any $\chi\in X^*(Z)$
\begin{eqnarray}
  \label{eq:filtrationGIT}
  \sharp\Phi(w_1,\chi)+\sharp\Phi(w_2,\chi)+\sharp \Phi(w_3,\chi)+
\sum_{\alpha\in \Phi(G/P,\chi)}\langle h,\alpha\rangle
=2 \sharp \Phi(G/P_\beta,\chi).
\end{eqnarray}
\end{theo}

\subsection{Degree of numerical instability}

We first compute explicitly $\mu^\Li(\monx,\tau)$ in terms of the
Peterson decomposition.
 
\begin{lemma}
\label{lem:calculmu}
Recall that $\monx\in \mvar $ and $\tau\in X_*(T)$ is dominant.
Let $h\in X_*(T)$ and $w_i\in W$ (for
 $i=1,2,3$) such that $\monx$ belongs to $S_{-h}\times Uw_1^{-1}B/B\times
 Uw_2^{-1}B/B\times Uw_3^{-1}B/B$. 
    Then
$$
\mu^\Li(\monx,\tau)= \bl(h,\tau)+\sum_{i=1}^3\langle w_i\tau,\lambda_i\rangle.
$$
\end{lemma}

\begin{proof}
 The group $T$ acts on the fiber over $B/B$ in $\Li(\lambda_i)$ with
 weight $-\lambda_i$.
It follows that it acts on the fiber over $w_i^{-1}B/B$ in $\Li(\lambda_i)$ with
 weight $-w_i^{-1}\lambda_i$.

Similarly, $\hat\lh$ acts on  the fiber over the base point of $\Gra$
in $\Li(\Lambda)$
with weight $-\Lambda$.
Then $\hat\lh$ acts on the fiber over $L_{-h}$ in $\Li(\bl\Lambda)$ by
the weight $-\bl T_{-h}(\Lambda)$ (with notation of Section~\ref{sec:Waff}).
 But $-\bl T_{-h}(\Lambda)=\bl\Lambda-(\bl h,\square)-\frac \bl 2
 (h,h)\delta$ and $T$ acts on the fiber over $L_{-h}$ in
 $\Li(\bl\Lambda)$ by
weight $-\bl(h,\square)$.
The lemma follows.
\end{proof}

\bigskip
We set
$$
M^\Li(x)=\sup_{
  \begin{array}{l}
    \tau\in X_*(T)_\RR^+{\rm\ nontrivial}\\
g\in L^{<0}_{alg}G
  \end{array}
}\frac{\mu^\Li(gx,\tau)}{\Vert\tau\Vert}.
$$

\begin{prop}\label{prop:Mbiendef}
  Assume that $\monx$ is not numericaly  semistable. Then $M^\Li(\monx)$ is finite and
  there exist $g\in L^{<0}_{alg}G$ and $\tau\in X_*(T)^+$ nontrivial such
  that
$$M^\Li(\monx)=\frac{\mu^\Li(g\monx,\tau)}{\Vert\tau\Vert}.
$$
\end{prop}

\begin{proof}
   Let $h_0\in X_*(T)^+$ such that the projection of  $\monx$ on
   $\Gra$ belongs to $\Gra^{h_0}$.
Let $h\in X_*(T)$ such that $S_{-h}\cap \Gra^{h_0}\neq \emptyset$. 
We claim that 
$$
h\leq -w_0h_0.
$$
Let $y\in S_{-h}\cap \Gra^{h_0}$.
By formula~\eqref{eq:GrahBmoins}, there exists $w\in W$ such that $y\in
\Iw^-L_{wh_0}$. Then, by~\eqref{eq:BmoinsBB}
$\lim_{t\to\infty}(\delta^\vee+\rho^\vee)(t)y=L_{wh_0}$.
Since $S_{-h}$ is $\hat T$-stable, $L_{wh_0}$ belongs to
$\overline{S_{-h}}$.
By~\eqref{eq:Shbar}, this implies that $wh_0\leq -h'$.
Since $h_0$ is dominant  $w_0h_0\leq wh_0$.
The claim follows.\\

Denote by
$
\Theta$ the set of $(-h,w_1^{-1},w_2^{-1},w_3^{-1})\in X_*(T)\times
W^3$ such that 
$$ L^{<0}_{alg}G.\monx\cap
(S_{-h}\times U
w_1^{-1}B/B\times Uw_2^{-1}B/B\times Uw_3^{-1}B/B)\neq\emptyset.
$$
By Lemma~\ref{lem:calculmu}, we have
$$
M^\Li(\monx)=\sup_{
  \begin{array}{l}
    \tau\in X_*(T)_\RR^+{\rm\ s.t.\ }\Vert\tau\Vert=1\\
(-h,w_1^{-1},w_2^{-1},w_3^{-1})\in \Theta
  \end{array}
}\bl(h,\tau)+\sum_{i=1}^3\langle w_i\tau,\lambda_i\rangle.
$$
For such a $h$, the claim asserts that $h\leq -w_0h_0$. In particular,
for any dominant $\tau\in X_*(T)_\RR$, we have 
$(h,\tau)\leq (-w_0h_0,\tau)$.
It follows that $M^\Li(\monx)$ is finite.\\

By the above argument, there exists  sequences $
(-h_n,(w^n_1)^{-1},(w^n_2)^{-1},(w^n_3)^{-1})\in \Theta$ 
and $\tau_n\in X_*(T)_\RR$ such that 
$$
\lim_{n\to\infty}\bl(h_n,\tau)+\sum_{i=1}^3\langle
w^n_i\tau_n,\lambda_i\rangle=
M^\Li(\monx)
\quad{\rm and}\quad \Vert\tau_n\Vert=1.
$$
By extracting a subsequence, one may assume that each $w_i^n$ is
constant (equal to $w_i$) and that $\tau_n$ tends to $\tau_0\in
X_*(T)_\RR^+$. Then, $(-h_n,w_1^{-1},w_2^{-1},w_3^{-1})\in \Theta$,
$\Vert\tau_0\Vert=1$ and
$$
\lim_{n\to\infty}\bl(h_n,\tau_0)+\sum_{i=1}^3\langle
w_i\tau_0,\lambda_i\rangle=
M^\Li(\monx).
$$
Set $M'=\frac 1 \bl (M^\Li(\monx)-\sum_{i=1}^3\langle
w_i\tau_0,\lambda_i\rangle-1)$.
Since $\tau_0$ is dominant, for $h\leq -w_0h_0$,  the function $h\mapsto (h,\tau_0)$ takes
only finitely many values greater than $M'$.
Hence the sequence $\bl(h_n,\tau_0)+\sum_{i=1}^3\langle
w_i\tau_0,\lambda_i\rangle$ is stationary and there exists
$(-h,w_1^{-1},w_2^{-1},w_3^{-1})\in \Theta$ such that 
$$
M^\Li(\monx)=\bl(h,\tau_0)+\sum_{i=1}^3\langle
w_i\tau_0,\lambda_i\rangle.
$$

\bigskip
Let $\Face$ be the face of $X_*(T)_\RR^+$ containing $\tau_0$ in its
relative interior. 
Then 
$$
M^\Li(\monx)=\sup_{\tau\in \Face{\rm\ s.t.\ }\Vert\tau\Vert=1}
\bl(h,\tau_0)+\sum_{i=1}^3\langle
w_i\tau_0,\lambda_i\rangle.
$$
Since the linear form $\tau \mapsto \bl(h,\tau_0)+\sum_{i=1}^3\langle
w_i\tau_0,\lambda_i\rangle$ and the cone $\Face$ are rational, this
supremum is reached on a rational half line.  It follows that
$\tau_0=\frac{\tau_0'}{\Vert\tau_0'\Vert}$ for some rational $\tau_0'$. 
\end{proof}

\subsection{Relation with parabolic bundles}
\label{sec:parb}

A {\it flagged bundle} $(\parE,\xi_1,\xi_2,\xi_3)$ on $\PP^1$ at the
three marked points $p_1$, $p_2$ and $p_3$ is the given of a principal
$G$-bundle $\parE$ on $\PP^1$ and three parabolic reductions
$\xi_i\in\parE_{p_i}/B$ at the three points $p_i$.  
Let us recall how to associate to any point of $\mvar$ a flagged
bundle on $\PP^1$. 
Assume that $\{p_1,p_2,p_3\}\cap\{0,\infty\}$ is empty.
Let $\iota_0\,:\,\CC\longto \PP^1, z\longmapsto [z,1]$ and 
$\iota_\infty\,:\,\CC\longto \PP^1, z\longmapsto [1,z]$ denote the two
open embeddings. 
Their images cover $\PP^1$, and for any $z\in\CC^*$,
$\iota_0(z)=\iota_\infty(z^{-1})$.
Let $g\in L_{alg}G$. Thinking about $g$ as a transition function on
$\iota_0(\CC)\cap\iota_\infty(\CC)$, we get a principal $G$-bundle
$\parE$ with two trivializations $\tilde\iota_0\,:\,\CC\times G\longto
\parE$ over $\PP^1-\{\infty\}$ and 
$\tilde\iota_\infty\,:\,\CC\times G\longto
\parE$ over $\PP^1-\{0\}$.
Moreover, for any $z\in\CC^*$ and $h\in G$, we have
$$
\tilde\iota_0(z,h)=\tilde\iota_\infty(z^{-1},g(z)h).
$$
Consider also the two sections
$\sigma_0$ and $\sigma_\infty$ defined respectively on
$\PP^1-\{\infty\}$ and $\PP^1-\{0\}$ by
$$
\sigma_0(\iota_0(z))=\tilde\iota_0(z,e)\quad{\rm and}\quad
\sigma_\infty(\iota_\infty(z))=\tilde\iota_\infty(z,e),
$$ 
for any $z\in\CC$.
The map $g\longmapsto(\parE,\sigma_0,\sigma_\infty)$ is a bijection
from $L_{alg}G$ to the set of principal bundles on $\PP^1$ endowed
with two sections. 

Let $g_1\in L_{alg}^{<0}G$ and $g_2\in L_{alg}^{>0}G$. 
Let $\parE'$, $\tilde\iota_0'$, $\tilde\iota_\infty'$, $\sigma_0'$ and 
$\sigma_\infty'$ be as above when $g$ is replaced by $g_1gg_2^{-1}$.
Then, there exits an isomorphism $\Theta\,:\,\parE\longto\parE$ of
principal $G$-bundles such that
$$
\Theta(\tiota_0(z,h))=\tiota_0'(z,g_2(z)h)\quad{\rm and}\quad
\Theta(\tiota_\infty(z,h))=\tiota_\infty'(z,g_1(z^{-1})h),
$$
for any $z\in\CC$ and $h\in G$.
Moreover
$$
\Theta\circ\sigma_0\circ\iota_0=\sigma_0'\circ\iota_0. g_2 \quad{\rm and}\quad
\Theta\circ\sigma_\infty\circ\iota_\infty=\sigma_\infty'\circ\iota_\infty
g_1.
$$
In other words, $g_1gg_2^{-1}$ corresponds to
$(\parE,\sigma_0g_2^{-1},\sigma_\infty g_1^{-1})$.

Now $\Gra=L_{alg}G/L_{alg}^{>0}G$ corresponds to the set of pairs
$(\parE,\sigma_\infty)$.
Let 
$$\monx=(gL_{alg}^{>0}G/L_{alg}^{>0}G,g_1B/B,g_2B/B,g_3B/B)\in\mvar.$$
Let $(\parE,\sigma_\infty)$ corresponding to $ gL_{alg}^{>0}G/L_{alg}^{>0}G$.
Consider, for any $i=1,2,3$,  the point $\sigma_\infty(p_i)g_iB/B$ in
$\parE_{p_i}/B$; it is a parabolic reduction $\xi_i$ at $p_i$. One checks that 
two points $\monx$ and $\monx'$ in $\mvar$ induces the same flagged
bundle $(\parE,\xi_1,\xi_2,\xi_3)$ if and only if they belong to the
same $L_{alg}^{<0}G$-orbit.

The given of $\monx\in\mvar$ also determines a section
$\sigma_\infty\,:\,\PP^1-\{0\}\longto\parE$. This section induces a
section $\PP^1-\{0\}\longto\parE/B$ that extends  to a
parabolic reduction $\bar\sigma\,:\,\PP^1\longto \parE/B$.\\

\noindent{\bf Parabolic degree.}
Recall that  $(\lambda_1,\lambda_2,\lambda_3)\in(X^*(T))^3$,
$\bl\in\ZZ_{\geq 0}$  and $\Li=\Lila$ on $\mvar$. 
Let us explain how $\mu^\Li(\monx,\tau)$ can be expressed in terms
of the flagged bundle $(\parE,\xi_1,\xi_2,\xi_3)$ endowed with the
parabolic reduction $\bar\sigma$.

Fix $i=1,2$ or $3$. Both $\bar\sigma(p_i)$ and $\xi_i$ belong to
$\parE_{p_i}/B$. 
Fixing an identification $\parE_{p_i}\simeq G$ (which is equivariant
for the right $G$-actions), the pair
$(\bar\sigma(p_i),\xi_i)$ gives a point in $G/B\times G/B$.
The $G$-orbit of this point does not depend on the chosen
identification $\parE_{p_i}\simeq G$; in particular, it belongs to
$G.(B/B,w_i^{-1}B/B)$ for some well defined  $w_i\in W$.

Consider $(\tau,\square)\in X^*(T)$.
The parabolic reduction $\bar\sigma$ induces a principal $B$-bundle
$\parE_B$.
We denote by  $\CC_\tau$ the one-dimensional representation of $B$
associated to the character $(\tau,\square)$ of $B$.
We can define the line bundle $\parE_B\times_B\CC_{\tau}$ on
$\PP^1$. Its degree $\deg(\parE_B\times_B\CC_{\tau})$ belongs to $\ZZ$.

The {\it parabolic degree relatively to $\Li$} is defined by 
\begin{eqnarray}
  \label{eq:defpardeg}
  \pardeg(\parE, \xi_1,\xi_2,\xi_3,\bar\sigma,\tau)=\bl\deg(\parE_B\times_B\CC_{\tau})+\sum_{i=1}^3\langle w_i^{-1}\lambda_i,\tau\rangle.
\end{eqnarray}
\begin{lemma}\label{lem:pardegmu}
  With above notation, we have
$$
\mu^\Li(\monx,\tau)=\pardeg(\parE, \xi_1,\xi_2,\xi_3,\bar\sigma,\tau).
$$
\end{lemma}

\begin{proof}
  Let $h\in X_*(T)$ and $v_i\in W$ (for
 $i=1,2,3$) such that $\monx$ belongs to $S_{-h}\times Uv_1^{-1}B/B\times
 Uv_2^{-1}B/B\times Uv_3^{-1}B/B$. With Lemma~\ref{lem:calculmu}, it
 is sufficient to prove that 
$
(h,\tau)=\deg(\parE_B\times_B\CC_{\tau}),
$
and that $v_iW_{P(\tau)}=w_iW_{P(\tau)}$, for any $i=1,2,3$.
These are direct verifications. 
\end{proof}

\bigskip
Recall from e.g.\ \cite{HeinlothSchmitt:ssparbund}, that $(\parE, \xi_1,\xi_2,\xi_3)$ is said to
be {\it semistable  relatively to $\Li$} if and only
if for any dominant $\tau\in X_*(T)$ and any
parabolic reduction $\bar\sigma\,:\,\PP^1\longto \parE/B$ we have
$$
\pardeg(\parE, \xi_1,\xi_2,\xi_3,\bar\sigma,\tau)\leq 0.
$$

\begin{coro}
\label{cor:ss=ss}
Fix $\monx\in\mvar$ and the corresponding flagged principal bundle $
(\parE,\xi_1,\xi_2,\xi_3)$.
Then
$\monx$ is numericaly  semistable in the sense of Definition~\ref{sec:defxss} if and
only if the flagged principal $(\parE,\xi_1,\xi_2,\xi_3)$ is
semistable relatively to $\Li$.
\end{coro}

\subsection{Generic toric reduction}

Let $(\parE,\xi_1,\xi_2,\xi_3)$ be a flagged principal bundle.
Let $\Omega$ be a nonempty open subset of $\PP^1$ and
$\eta\,:\,\Omega\longto \parE/T$ be a reduction defined on $\Omega$.
Let $\tau\in X_*(T)$. 
Since $T\subset P(\tau)$, we have a quotient map
$\parE/T\longto\parE/P(\tau)$. 
Hence $\eta$ induces a reduction $\Omega\longto\parE/P(\tau)$ that
extends to $\bar\sigma\,:\,\PP^1\longto\parE/P(\tau)$.
Consider the map
$$
\begin{array}{cccl}
  \mu_\eta\,:&X_*(T)&\longto&\ZZ\\
&\tau&\longmapsto&\pardeg(\parE,\xi_1,\xi_2,\xi_3,\bar\sigma,\tau).
\end{array}
$$
Note that in \eqref{eq:defpardeg}, we replace
$\parE_B\times_B\CC_{\tau}$ by $\parE_{P(\tau)}\times_{P(\tau)}\CC_{\tau}$.
Since $P(\tau)$ and $\bar\sigma$ only depends on the signs
of the $\langle \tau,\alpha\rangle$ for $\alpha\in\Phi$, the map
$\mu_\eta$ is piecewise linear. In particular, it extends to a
positively homogeneous, continuous, piecewise linear function from
$X_*(T)\otimes\RR$ to $\RR$. This extension is still denoted by
$\mu_\eta$.
The following proposition will play an important role.

\begin{prop}\label{prop:muetaconvex}
  Assume that  $\frac{\lambda_1}\bl$, $\frac{\lambda_2}\bl$ and $\frac{\lambda_3}\bl$ belong to the
  alcove $\alc^*$. 
Then the map $\mu_\eta$ is convex.
\end{prop}

\begin{proof}
  Recall that the cones of the Weyl fan  are the
subsets of $\tau\in X_*(T)\otimes\RR$ such that for each
$\alpha\in\Phi$,  the product
$\langle\tau,\alpha\rangle$ is fixed to be negative, positive or zero.
The parabolic subgroup $P(\tau)$ only depends on the cone of the Weyl
fan containing $\tau$ in its relative interior. Now, the
  formula~\eqref{eq:defpardeg} and Lemma~\ref{lem:pardegmu} show that 
the restriction of $\mu_\eta$ to any such cone is linear.
 Then it is sufficient to check convexity when one goes from any
 chamber to an adjacent one. 
To simplify notations, we assume that one of these two chambers is the
dominant one $X_*(T)_\RR^+$. The other one is $s_\alpha. X_*(T)_\RR^+$
for some simple root $\alpha$. The minimal parabolic subgroup
$P^\alpha$ associated to some simple root $\alpha$ is the closure of $Bs_\alpha Bs_\alpha$. 
Let $\mu\in \Hom(X_*(T),\QQ)$ (resp. $\mu'\in \Hom(X_*(T),\QQ)$) whose
the restriction to $X_*(T)_\RR^+$ (resp. $s_\alpha.X_*(T)_\RR^+$) is
equal to $\mu_\eta$(resp $\mu'_\eta$).
Recall that $\alpha^\vee$ is orthogonal to the span of
$X_*(T)_\RR^+\cap (s_\alpha X_*(T)_\RR^+)$ and
oriented toward $X_*(T)_\RR^+$.
The convexity on the union of these two chambers is equivalent to the
following inequality
\begin{eqnarray}
  \label{eq:convexadjchambers}
  \mu(\alpha^\vee)\geq\mu'(\alpha^\vee).
\end{eqnarray}

Let $L^\alpha$ denote the Levi subgroup of $P^\alpha$ containing $T$
and let $R^u(P^\alpha)$ denote the unipotent radical of $P^\alpha$.
Consider the reduction
$\bar\sigma_\alpha\,:\,\PP^1\longto\parE/P^\alpha$ induced by $\eta$
and $\parE_{P^\alpha}$ be the associated principal $P^\alpha$-bundle.
Since $L^\alpha$
identifies with $P/R^u(P^\alpha)$, $\parE_{P^\alpha}/R^u(P^\alpha)$
is a principal $L^\alpha$-bundle on $\PP^1$ denoted by $\parE_{L^\alpha}$.
We are going to endow $\parE_{L^\alpha}$ with a flagged structure and
express  $\mu(\alpha^\vee)-\mu'(\alpha^\vee)$ in terms of this flagged
bundle.

Set $\bar\alpha=\frac 2 {(\alpha,\alpha)}\alpha$.
First, observe that $\parE_B\times_B\CC_{\bar\alpha}$
(resp. $\parE_{B'}\times_{B'}\CC_{\bar\alpha}$ ) identifies with $\parE_{B\cap
  L^\alpha}\times_{B\cap L^\alpha}\CC_{\bar\alpha}$ (resp.  $\parE_{B'\cap
  L^\alpha}\times_{B'\cap L^\alpha}\CC_{\bar\alpha}$).
In particular,
\begin{eqnarray}
  \label{eq:3}
  \deg(\parE_B\times_B\CC_{\bar\alpha}
)
-\deg(\parE_{B'}\times_{B'}\CC_{\bar\alpha}
)=
\deg(\parE_{B\cap
  L^\alpha}\times_{B\cap L^\alpha}\CC_{\bar\alpha}
)
-\deg(\parE_{B'\cap
  L^\alpha}\times_{B'\cap L^\alpha}\CC_{\bar\alpha}.
)
\end{eqnarray}

Fix $i=1,2$ or $3$. Choose an identification $\parE_{p_i}\simeq G$
such that $\bar\sigma_\alpha(p_i)$ corresponds with $P^\alpha$.
Let $B_i$ be the Borel subgroup $G$ associated to the flagged structure
at $p_i$. Then $(B_i\cap P^\alpha)/R^u(P^\alpha)$ is a Borel subgroup
of $L^\alpha$. Let $\xi^\alpha_i$ be the associated flagged structure at
$p_i$ for $\parE_{L^\alpha}$.
Now $(\parE_{L^\alpha},\xi^\alpha_1,\xi^\alpha_2,\xi^\alpha_3)$ is a
flagged principal $L^\alpha$-bundle.

Observe that $\parE/B$ and $\parE/B'$ identify canonically. Then,
$\bar\sigma_B$ and $\bar\sigma_{B'}$ are two sections of this
$G/B$-bundle.
We have the commutative diagram

\begin{center}
\begin{tikzpicture}
  \matrix (m) [matrix of math nodes,row sep=3em,column sep=4em,minimum width=2em]
  {&\parE/B\\
     \PP^1 & \parE/P_\alpha \\
     & \parE/B\\};
  \path[-stealth]  
    (m-2-1) edge node [above] {$\bar\sigma_\alpha$} (m-2-2)
    (m-2-1) edge node [above] {$\bar\sigma_\alpha$} (m-1-2)
    (m-2-1) edge node [below] {$\bar\sigma_\alpha$} (m-3-2)
    (m-1-2) edge node [right]{$\pi$} (m-2-2)
    (m-3-2) edge node [right]{$\pi$}(m-2-2);       
\end{tikzpicture} 
\end{center}
where the vertical maps $\pi$ are induced by the inclusion $B\subset
P^\alpha$.
By construction $\pi(\bar\sigma_B(p_i))= 
\pi(\bar\sigma_B(p_i))=\bar\sigma_\alpha(p_i)$. 
Recall that $w_i\in W$ is characterized by the relation
$(\bar\sigma_B(p_i),\xi_{p_i})\in G.(B/B,w_i^{-1}B/B)$.
 Similarly we defined $w_i'$. 
 Then, $w_i'$ is equal to either $w_i$ or $w_is_\alpha$.

\bigskip
Consider the $L^\alpha$-irreducible representation
$V_{\varpi_\alpha}(L^\alpha)$ of highest weight $\varpi_\alpha$. It
has dimension two, a unique $B\cap L^\alpha$-fixed line $d$ 
and  a unique $B'\cap L^\alpha$ fixed line $d'$. 
Consider the associated short exact sequence
\begin{center}
  
\begin{tikzpicture}
  \matrix (m) [matrix of math nodes,row sep=3em,column sep=4em,minimum width=2em]
  {0&d&V_{\varpi_\alpha}(L^\alpha)&V_{\varpi_\alpha}(L^\alpha)/d'&0\\};
  \path[-stealth]  
    (m-1-1) edge (m-1-2)  (m-1-2)  edge (m-1-3) (m-1-3)  edge (m-1-4) (m-1-4)  edge (m-1-5);       
\end{tikzpicture} 
\end{center}

Consider the associated morphisms of vector bundles on $\PP^1$:

\begin{tikzpicture}
  \matrix (m) [matrix of math nodes,row sep=3em,column sep=4em,minimum width=2em]
  {\parE_{L^\alpha\cap B}\times_{L^\alpha\cap B} d   &
 \parE_{L^\alpha}\times_{L^\alpha} V_{\varpi_\alpha}(L^\alpha)
&  \parE_{L^\alpha\cap B'}\times_{L^\alpha\cap B'} V_{\varpi_\alpha}(L^\alpha)/d'\\  };
 \path[-stealth]  
    (m-1-1) edge (m-1-2) (m-1-2)  edge (m-1-3);      
\end{tikzpicture} 

Set $\lambda_i^\alpha=w_i\lambda_i$, for $i=1,2,3$ (HERE $w_i\in W^{P^\alpha}$).
Let  $v_i=s_\alpha$ or $e$ denote the relative position of the flag on
$(\parE_{L^\alpha})_{p_i}$ and $\bar\sigma_B^\alpha(p_i)$. Similarly $v_i'$.
Inequality~\eqref{eq:convexadjchambers} is equivalent to 
$$
\begin{array}{ll}
  \deg( \parE_{L^\alpha\cap B}\times_{L^\alpha\cap B} d)+
\frac 1   2\sum_i\langle v_i\lambda_i^\alpha,\alpha^\vee\rangle&\geq\\
  \deg( \parE_{L^\alpha\cap B'}\times_{L^\alpha\cap B'}
    V_{\varpi_\alpha}(L^\alpha)/d)+
\frac 1 2\sum_i\langle v_i'\lambda_i^\alpha,\alpha^\vee\rangle\\
\end{array}.
$$
But $\langle v_i\lambda_i,\alpha^\vee\rangle-\langle
v_i'\lambda_i,\alpha^\vee\rangle$ is either equal to $0$, $\pm 2\langle
\lambda_i^\alpha,\alpha^\vee\rangle$.
The two cases $0$ and sign + are easy. Consider the last case.
The point is that in this case the morphism has to vanish at $p_i$. 
In particular, if this case occurs $d$ times (when $i$ runs over $\{1,2,3\})$
then $\deg( \parE_{L^\alpha\cap B}\times_{L^\alpha\cap B} d)-\deg(
  \parE_{L^\alpha\cap B'}\times_{L^\alpha\cap B'}
    V_{\varpi_\alpha}(L^\alpha)/d)\geq d$. 
But we have
$0\leq \langle \lambda_i^\alpha,\alpha^\vee\rangle\leq 1$.
Inequality \eqref{eq:convexadjchambers} follows.
\end{proof}


\bigskip{\bf The polytope $\poly_\eta$.}
Consider in $X_*(T)_\RR$ the fan $\Sigma$ whose the maximal cones are
the Weyl chambers. By formula~\eqref{eq:defpardeg} and
Lemma~\ref{lem:pardegmu}, the restriction of $\mu_\eta$ to any Weyl
chamber is linear. In particular, $\mu_\eta$ is piecewise linear.
Like in Section~\ref{sec:cvxe}, consider the associated polytope
$$
\poly_\eta=\{\chi\in X^*(T)_\RR\,:\, \forall\tau\in
X_*(T)\quad\langle\tau,\chi\rangle\geq\mu_\eta(\tau)\}.
$$
Let $B'$ be a Borel  subgroup of $G$ containing $T$ and let $C'$
denote the corresponding Weyl chamber in $X_*(T)_\RR$.
Let $\chi_{B'}$ be the only point in $X^*(T)_\QQ$ such that $\langle
\chi_{B'},\tau\rangle=\mu_\eta(\tau)$ for any $\tau$ in $C'$. Then
$\poly_\eta$ is the convex hull of the $\chi_{B'}$ for various
Borel subgroups $B'\supset T$.

Similarly the rays of $\Sigma$ correspond bijectively with the maximal
parabolic subgroups $P$ containing $T$. For any such parabolic
subgroup, let $\tau_P$ denote the unique indivisible one parameter
subgroup of $T$ such that $P=P(\tau_P)$. Then
$$
\poly_\eta=\{\chi\in X^*(T)_\RR\,:\, \forall {\rm \ maximal\ }P\supset
T\quad \langle
\chi,\tau_P\rangle\geq\mu_\eta(\tau_P)\}.
$$

\subsection{Canonical reduction}

Let $\monx\in\mvar$ be numerically unstable. Let $g\in L^{<0}_{alg}G $
and let $\tau_0$ be an
indivisible dominant one-parameter subgroup  of $T$ such that 
  $M^\Li(\monx)=\frac{\mu^\Li(g\monx,\tau)}{\Vert\tau_0\Vert}$.
To the point $g\monx$ corresponds a flagged bundle $\Eps$ with a section
$\sigma_\infty$ over $\PP^1-\{0\}$.
This section extends to a parabolic reduction
$\sigma^{P(\tau_0)}_{g\monx}\,:\,\PP^1\longto \parE/P(\tau_0)$.

\begin{prop}
\label{prop:unicitytau}
Assume that $\frac{\lambda_1}\bl$, $\frac{\lambda_2}\bl$ and $\frac{\lambda_3}\bl$ belong to the
  alcove $\alc^*$. 
Let $\monx\in\mvar$ and $(\parE,\xi_1,\xi_2,\xi_3)$ be the associated
flagged bundle. Assume that $\monx$ is unstable relatively to $\Li$.

Let $g_1$ and $g_2$ in  $L^{<0}_{alg}G$ and let $\tau_1$ and $\tau_2$ be
two dominant indivisible one parameter subgroups of $T$ such that
$$
M^\Li(\monx)=\frac{\mu^\Li(g_1\monx,\tau_1)}{\Vert\tau_1\Vert}=\frac{\mu^\Li(g_2\monx,\tau_2)}{\Vert\tau_2\Vert}.
$$
  Then
  \begin{enumerate}
  \item $\tau_1=\tau_2$ ; set $P=P(\tau_1)=P(\tau_2)$.
\item $g_2g_1^{-1}\in L_{alg}^{<0}P$.
\item  The two reductions $\bar\sigma_1$ and $\bar\sigma_2$ from
  $\PP^1$ to $\parE/P$ associated to $g_1\monx$ and $g_2\monx$
  respectively coincide.
This reduction is called the {\it canonical reduction} of $(\parE,\xi_1,\xi_2,\xi_3)$.
  \end{enumerate}
\end{prop}

\begin{proof}
Let $\sigma_1^B$ and $\sigma_2^B$ denote the parabolic reductions
associated to $g_1\monx$ and $g_2\monx$ respectively.
For $x\in\PP^1$, $\sigma_B(x)$ and $\sigma'_B(x)$ belong to
$\parE_x/B$. Choosing an identification  $\parE_x/B\simeq G/B$, we get
two point in $G/B$. The element $w\in W$ such that $G.(B/B,wB/B)$
does not depend on the identification and is denote by rel$(\sigma_B(x),\sigma_B'(x))$.
By the finiteness of the Bruhat decomposition, there exists $w\in W$
and a nonempty open subset $\Omega$ of $\PP^1$ such that
rel$(\sigma_B(x),\sigma_B'(x))=w$ for any $x\in\Omega$.
Set $B_w=wBw^{-1}$. Observe the $\parE/B$ and $\parE/B_w$ are
canonically isomorphic. Let $p,p_w\,:\,\parE\longto \parE/B$ be the two
projections induced respectively by the inclusions $T\subset B$ and
$T\subset B_w$. 
Up to changing $\Omega$, one may assume that $\parE$ is trivial on
$\Omega$. 
Then, there exists a reduction
$\eta\,:\,\Omega\longto\parE/T$ such that $\sigma_1^B=p\circ\eta$ and
$\sigma_2^B=p_w\circ\eta$.

By Lemma~\ref{lem:pardegmu}, we have
$$
\mu^\Li(g_1\monx,\tau_1)=\mu_\eta(\tau_1)\quad{\rm and}\quad
\mu^\Li(g_2\monx,\tau_2)=\mu_\eta(w\tau_2w^{-1}).
$$
In particular, $\sup_{\tau\in X_*(T) {\rm\ nontrivial}}\frac{\mu_\eta(\tau)}{\Vert\tau\Vert}=M^\Li(\monx)$.
By Proposition~\ref{prop:muetaconvex}, the function is convex. In
particular, it has a unique maximum on the unit sphere. Hence
$\tau_1=w\tau_2w^{-1}$. Since $\tau_1$ and $\tau_2$ are assumed to be dominant,
$\tau_1=\tau_2$.\\

As in the proposition set $P=P(\tau_1)$.
We have also proved that $\tau_2=w\tau_2w^{-1}$. Hence $w$ belongs to
$W_P$.
Then $p$ and $p_w$ induce the same map $q\,:\,\parE/T\longto\parE/P$.
Therefore,  $\sigma_1^B$ and $\sigma_2^B$ induce the same reduction
$\sigma\,:\,\PP^1\longto \parE/P$. The last assertion of the
proposition follows.\\

Let $\sigma_\infty$ be the section of $\parE$ on $\PP^1-\{0\}$
associated to $\monx$.
Then $g_1\monx$ and $g_2\monx$ correspond respectively to 
$\sigma_\infty g_1^{-1}$ and $\sigma_\infty g_2^{-1}$.
But, we just proved these two local trivialisations induce the same
section of $\parE/P$. 
Hence $g_1^{-1}(z)P/P=g_2^{-1}(z)P/P$ for any $z$. The second
assertion is proved.
\end{proof}

\bigskip
\begin{defin}
Let $\monx\in\mvar$ be unstable relatively to $\Li$. 
Let $(\parE,\xi_1,\xi_2,\xi_3)$ be the associated flagged principal
bundle.
Let $\tau_0$ denote the  dominant one parameter subgroup of $T$ satisfying
Proposition~\ref{prop:unicitytau}.
Set $P=P(\tau_0)$.
Let $\sigma\,:\,\PP^1\longto \parE/P$ be the canonical reduction of
$(\parE,\xi_1,\xi_2,\xi_3)$ and $\parE_P$ the associated principal $P$-bundle.
For $i=1,2$ and $3$, let $w_i\in W^P$ denote the relative position of
$(\xi_i,\sigma(p_i))$.
Finally, we define a $\ZZ$-linear map
$$
\begin{array}{lccc}
  h\,:&X^*(P)&\longto&\ZZ\\
&\chi&\longmapsto&\deg(\parE_P\times_P\CC_\chi).
\end{array}
$$  
The {\it Harder-Narashiman type (HN-type for short)} of 
  $\monx$ (or of $(\parE,\xi_1,\xi_2,\xi_3)$) is the uple $(\tau_0,P,h,w_1,w_2,w_3)$.
\end{defin}

\bigskip {\bf A characterization of the canonical reduction.}
 Let $P\supset T$
be a parabolic subgroup, let $R^u(P)$ denote its unipotent radical  and let $L$ denote its Levi subgroup
containing $T$.

Let $(\parE,\xi_1,\xi_2,\xi_3)$ be a flagged bundle. Let $\sigma\,:\,\PP^1\longto\parE/P$ be a parabolic
reduction.
Let $\parE_P\subset \parE$ denote the principal $P$-subdundle associated
to $\sigma$. Then the quotient $\parE_P/R^u(P)$ is a principal
$L$-bundle. Consider a marked point $p_i$ and choose an identification
of the fiber $\parE_{p_i}$ with $G$ (as torsor). Then $\sigma(p_i)$
determines a parabolic subgroup $P'$ of $G$ conjugated to
$P$. Similarly $\xi_i$ determines a Borel subgroup $B'$ of $G$. 
Then $(P'\cap B')/R^u(P')$ is a Borel subgroup of $P'/R^u(P')$. This
Borel subgroup (which is independent on the choice) can be chosen as a
flag $\xi_i^L$ over $p_i$ in $\parE_P/R^u(P)$.
Then $(\parE_P/R^u(P),\xi_1^L,\xi_2^L,\xi_3^L)$ is a flagged
$L$-bundle over $(\PP^1,p_1,p_2,p_3)$.

Assume now that $(\parE,\xi_1,\xi_2,\xi_3)$ and the parabolic
reduction $\sigma$ come from $\monx\in \mvar$.
Let $\tau\in X_*(T)$ such that $P=P(\tau)$.
 Set $\monx_0=\lim_{s\to  0}\tau(s)\monx$.
It belongs to the fixed point set $\mvar^\tau$.
Each irreducible component of $(G/B)^\tau$ contains a unique $B\cap L$
fixed point and so identifies canonically with $L/(B\cap L)$. 
On the other hand, $L_{alg} L$ acts transitively on $\Gra^\tau$ that
identifies with the affine grassmannian $\Gra(L)$ of the group $L$.
Using these identifications, the point $\monx_0$ gives a point
$\monx_0'$ of $\Gra(L)\times (L/B\cap L)^3$.
Hence $\monx_0$ determines a flagged principal $L$-bundle. 
This bundle is $(\parE_P/R^u(P),\xi_1^L,\xi_2^L,\xi_3^L)$. 

\bigskip
Here, and like before, $\tau$ is a one parameter subgroup of $T$, 
$L$ is the centralizer of the image of $\tau$ and $P$ is the
associated parabolic subgroup. 
Since $Z(L)$ is contained in $T$, $X_*(Z(L))$ is contained in
$X_*(T)$. It is
$\oplus_{\alpha\in\Delta-\Delta_P}\ZZ\varpi_{\alpha^\vee}$.
In particular
$\tau=\sum_{\alpha\in\Delta-\Delta_P}n_\alpha\varpi_{\alpha^\vee}$,
for some integers $n_\alpha$.
The restriction map $X^*(L)\longto X^*(T)$ is injective, and 
$X^*(L)$ identifies with $\oplus_{\alpha\in\Delta-\Delta_P}\ZZ\varpi_{\alpha}$.
Consider $(\tau,\square)\in X^*(T)\otimes \QQ$.
We have:
$(\tau,\square)=\sum_{\alpha\in\Delta-\Delta_P}n_\alpha(\varpi_{\alpha^\vee},\square)=\sum_{\alpha\in\Delta-\Delta_P}n_\alpha\frac
2 {(\alpha,\alpha)}\varpi_{\alpha}$.
In particular, $(\tau,\square)$ belongs to $X^*(L)$.

\begin{prop}
  \label{prop:adapLss}
Assume that $\frac{\lambda_1}\bl$, $\frac{\lambda_2}\bl$ and $\frac{\lambda_3}\bl$ belong to the
  alcove $\alc^*$. 
Let $\monx\in\mvar$ be unstable and $\tau$ be a dominant
one parameter subgroup of $T$. Set $P=P(\tau)$, $L=G^{\tau}$ and
$\sigma\,:\,\PP^1\longto \parE/P$ be the parabolic reduction
associated to $\monx$ and $P$.
Consider the flagged principal $L$ bundle $(\parE_P/R^u(P),\xi_1^L,\xi_2^L,\xi_3^L)$.

The following are  equivalent
\begin{enumerate}
\item $M^\Li(\monx)=\frac{\mu^\Li(\monx,\tau)}{\Vert\tau\Vert}$;
\item  $\parE_P/R^u(P)$ is semistable for $L$ relatively to the line
  bundle $\Li\otimes -\mu^\Li(\monx,\tau)\frac{(\tau,\square)}{\Vert\tau\Vert}$.
\end{enumerate}
\end{prop}

\begin{proof}
Set $\Li'=\Li\otimes
-\mu^\Li(\monx,\tau)\frac{(\tau,\square)}{\Vert\tau\Vert}$ and 
$\chi=\mu^\Li(\monx,\tau)\frac{(\tau,\square)}{\Vert\tau\Vert}$.
  Assume that
  $M^\Li(\monx)=\frac{\mu^\Li(\monx,\tau)}{\Vert\tau\Vert}$.
Let $\zeta\in X_*(T)$ and $l\in L^{>0}_{alg}L$. We have to prove that
\begin{eqnarray}
  \label{eq:ineqLzeta}
  \mu^\Li(l\monx,\zeta)\leq \langle \chi,\tau\rangle.
\end{eqnarray}
Consider the parabolic reduction
$\sigma_\zeta\,:\,\PP^1\longto \parE/Q$ induced by $l\monx$ and
$\zeta$. 
By construction, there exists an embedding $Q\subset P$ such that if
$p\,:\,\parE/Q\longto \parE/P$ denotes the corresponding projection,
we have $p\circ\sigma_\zeta=\sigma$.

Consider now a generic reduction $\eta$ to $\parE/T$ and
$q\,:\,\parE/T\longto \parE/Q$ such that $\sigma_\zeta=q\circ\eta$.
Consider the convex function $\mu_\eta$ and the polytope $\poly_\eta$.
Since $\frac{\mu_\eta(\tau)}{\Vert\tau\Vert}=M^\Li(\eta)$,
the point $\chi$
belongs to $\poly_\eta$.
Recall that $\Face_\tau$ denote the face of $\poly_\tau$ corresponding
to the inequality $\langle\tau,\square\rangle\geq
\mu^\Li(\monx,\tau)$. 
This face is the polytope of $\parE_P/R^u(P)$ for $\eta$ relatively to
$\Li$.
Hence the polytope of $\parE_P/R^u(P)$ for $\eta$ relatively to
$\Li'$
is $\Face-\chi$. 
It contains $0$. Inequality~\eqref{eq:ineqLzeta} follows.

Conversely, assume that $\parE_P/R^u(P)$ is semistable for $L$ relatively to the line
  bundle $\Li'$.
Consider  a generic reduction $\eta$ to $\parE/T$ and
$p_P\,:\,\parE/T\longto \parE/P$ such that $\sigma=p_P\circ\eta$.
By the usual argument it is sufficient to prove that 
\begin{eqnarray}
  \label{eq:MLetatoprove}
  M^\Li(\eta)=\frac{\mu_\eta(\tau)}{\Vert\tau\Vert}.
\end{eqnarray}
The polytope $\poly_L$ of $\parE/R^u(P)$ for $\eta$ relatively to $\Li'$ is the
convex hull of the points $\chi_B-\chi$ for various $B$ such that
$T\subset B\subset P$.
This polytope is contained in $\poly_\eta-\chi$ which is the convex
hull of the points $\chi_B$ for various $B\supset T$.
By assumption, $0$ belongs to $\poly_L$. 
Hence $\chi$ belongs to $\poly_\eta$. Equality~\eqref{eq:MLetatoprove}
 follows.
\end{proof}

\bigskip
The following proposition is analogous to \cite[Theorem~9.3]{Ne2} or
\cite[Proposition~1.9]{RaRa}.

\begin{prop}
  \label{prop:limadap}
Assume that $\frac{\lambda_1}\bl$, $\frac{\lambda_2}\bl$ and $\frac{\lambda_3}\bl$ belong to the
  alcove $\alc^*$. 
Let $\monx\in\mvar$ be unstable.
Let $g$  in  $L^{<0}_{alg}G$ and let $\tau$ be
the dominant indivisible one parameter subgroup of $T$ such that
$$
M^\Li(\monx)=\frac{\mu^\Li(gx,\tau)}{\Vert\tau\Vert}.
$$
Set $\monx_0=\lim_{s\to 0}\tau(s)g\monx$. Then
\begin{enumerate}
\item $M^\Li(\monx_0)=M^\Li(\monx)$;
\item $M^\Li(\monx_0)=\frac{\mu^\Li(\monx_0,\tau)}{\Vert\tau\Vert}.$
\end{enumerate}
\end{prop}

\begin{proof}
  Since $\mu^\Li(\monx,\tau)=\mu^\Li(\monx_0,\tau)$, the second
  assertion implies the first one. 
By Proposition~\ref{prop:adapLss}, applied to $\monx=\monx_0$, the second assertion is equivalent
to the fact that $\parE/R^u(P)$ is semistable relatively to
$\Li\otimes
-\mu^\Li(\monx,\tau)\frac{(\tau,\square)}{\Vert\tau\Vert}$.
But, by Proposition~\ref{prop:adapLss}, applied to $\monx$, this is true.
\end{proof}

\subsection{The set of numericaly semistable points}

For later use let us state the following well known result.

\begin{lemma}\label{lem:ssopen}
  Assume that $\frac{\lambda_1}\bl$, $\frac{\lambda_2}\bl$ and $\frac{\lambda_3}\bl$ belong to the
  alcove $\alc^*$. Then $\mvar^{\rm nss}(\Lila)$ is open in $\mvar$.
\end{lemma}

\subsection{The open stratum}

In this section, we assume that no point in $\mvar$ is numericaly  semistable
relatively to $\Li$; that is that $M^\Li(\monx)>0$ for any
$\monx\in\mvar$.
Set
$$
d_0=\inf_{\monx\in \mvar}M^\Li(\monx),
$$
and 
$$
\mvar^\circ(\Li)=\{\monx\in\mvar\,:\,M^\Li(\monx)=d_0\}.
$$
This subset of $\mvar$ is called the {\it open stratum}. This term is
justified by the following proposition.

\begin{prop}
  \label{prop:openstratum}
  \begin{enumerate}
  \item The set $\mvar^\circ(\Li)$ is open and nonempty.
\item For any $\monx$ and $\mony$ in $\mvar^\circ(\Li)$, $M^\Li(\monx)=M^\Li(\mony)$.
  \item All points in  $\mvar^\circ(\Li)$ have the same indivisible
    dominant adapted one parameter subgroup of $T$. Let $\tau^\circ$
    denote this 1-PS.
\item Set $P=P(\tau^\circ)$. There exists an $(L_{alg}P)_0\times
  P^3$-orbit in $\mvar$ such that for any $\monx\in\mvar^\circ(\Li)$
  and any $g\in L_{alg}^{<0}G$ such that
  $\mu^\Li(g\monx,\tau^\circ)=\Vert\tau^\circ\Vert. d_0$, we have
  $g\monx\in C^+$.
  \end{enumerate}
\end{prop}

\begin{proof}
By the valuative criterion of openness, to prove the openness of
$\mvar^\circ(\Li)$ it is sufficient to prove the
following lemma.

\begin{lemma}
  Let $R$ be a discrete valuation ring and set $S={\rm Spec}(R)$. Let
  $\eta$ denote the generic point of $S$ and let $0$ denote the
  special one. 
Let $\parE$ be a flagged bundle on $\PP^1\times S$.

Then $M^\Li(\parE_0)\geq M^\Li(\parE_\eta)$.
\end{lemma}

The Behrend's proof of \cite[Proposition 7.1.3]{Behrend:thesis}
applies here. His proof also shows the end of the proposition.



Another useful reference is \cite[Proposition 2]{Heinloth:semistable}.
\end{proof}

\section{Gromov-Witten invariants and affine grassmannian}

\subsection{The homogeneous space $L_{alg}^{<0}G/L_{alg}^{<0}P$}

Let $P$ be a standard parabolic subgroup of $G$.
Let $\ud\in \Hom(X^*(P),\ZZ)$. 
We denote by $\Mor(\PP^1,G/P,\ud)$ the set of regular maps from $\PP^1$
to $G/P$ of degree $\ud$. It is empty or a quasiprojective variety. 
The disjoint union of the $\Mor(\PP^1,G/P,\ud)$ when $\ud$ runs over
$\Hom(X^*(P),\ZZ)$ is denoted by $\Mor(\PP^1,G/P)$.

Let  $g\in L_{alg}^{<0}G$. Then $g\circ\iota_\infty^{-1}$ (with
notation of Section~\ref{sec:parb}) is a
regular map from $\PP^1-\{0\}$ to $G$. By composition with
the projection $G\longto G/P$, one obtains a regular map from
 $\PP^1-\{0\}$ to $G/P$. Since $G/P$ is proper, this maps extend to
 $\PP^1$.
Let $\Theta(g)\in \Mor(\PP^1,G/P)$ denotes this map.
Observe that, for any $g,g'\in L_{alg}^{<0}G$, $\Theta(g)=\Theta(g')$
is and only if $g'^{-1}g\in L_{alg}^{<0}P$. Hence we just construct an
injective map
$$
\begin{array}{cccc}
  \Theta\,:&L_{alg}^{<0}G/L_{alg}^{<0}P&\longto&\Mor(\PP^1,G/P)\\
&gL_{alg}^{<0}P&\longmapsto&\Theta(g).
\end{array}
$$
Fix $\gamma\in\Mor(\PP^1,G/P)$. Since any $P$-principal
bundle on $\PP^1-\{0\}$ is trivial, the restriction of $\gamma$ raises
to $G$. Hence $\gamma$  belongs to the image of $\Theta$ that is surjective.

\bigskip
Recall that 
$$
L_{alg}G=\bccup_{h\in \oplus_{\alpha\in\Delta-\Delta_P}\ZZ\alpha^\vee}L_{alg}^{>0}Gz^{h}(L_{alg}P)_0.
$$

\begin{lemma}
  \label{lem:degdec}
Let $g\in L_{alg}^{<0}G$ and $h\in
\oplus_{\alpha\in\Delta-\Delta_P}\ZZ\alpha^\vee\simeq\Hom(X^*(P),\ZZ)$. 
Then $\Theta(gL_{alg}^{<0}P)$ has degree $h$ if and only if $g$
belongs to $L_{alg}^{>0}G z^{h}(L_{alg}P)_0$.
\end{lemma}

\begin{proof}
  By the decomposition just before the lemma, it is sufficient to
  prove that if $g$
belongs to $L_{alg}^{>0}G z^{h}(L_{alg}P)_0$ then
$\Theta(gL_{alg}^{<0}P)$ has degree $h$.

By Theorem~\ref{th:decPeterson}, there exist $w'\in W$ and 
$h'\in X_*(T\cap L^{\rm ss})$ such that $g$ belongs to
 $\Iw^- w' z^{h'}L_{alg}U$. Similarly,  there exist $w''\in W$ and 
$h''\in X_*(T\cap L^{\rm ss})$ such that $g$ belongs to
 $\Iw w'' z^{h+h''}L_{alg}U$. 
By Lemma~\ref{lem:degphi}, the curve $\Theta(g)\,:\,\PP^1\longto
G/B$ associated to $g$ has degree $h+h'-h''$. Hence
$\Theta(gL_{alg}^{<0}P)$ has degree $h$. 
\end{proof}

\subsection{Gromov-Witten invariants as ``degree''}

Fix  $w_1$, $w_2$, $w_3$ in $W^P$. With the notation of the
introduction, 
fix  also $\ud=\sum_{\beta\in\Delta-\Delta_P}d_\beta\sigma_{s_\beta}^*$
for some $d_\beta\in \ZZ_{\geq 0}$. 
Set $h=\sum_{\beta\in\Delta-\Delta_P}d_\beta\beta^\vee\in X_*(T)$.
Consider 
$$
C=L_{alg}L^{\rm ss}L_{-h}\times Lw_1^{-1}B/B \times Lw_2^{-1}B/B \times Lw_3^{-1}B/B
$$
and 
$$
C^+=(L_{alg}P)_0L_{-h}\times Pw_1^{-1}B/B \times Pw_2^{-1}B/B \times Pw_3^{-1}B/B.
$$
Observe that $L_{alg}^{<0}P$ is contained in $(L_{alg}P)_0$. 
In particular $C^+$ is stable by the action of $L_{alg}^{<0}P$.
Consider on 
$L_{alg}^{<0}G\times C^+$ the action of $L_{alg}^{<0}P$ given by the
formula 
$p.(g,\monx)=(gp^{-1},p\monx)$. This action is free and the quotient
set is denote by
 $L_{alg}^{<0}G\times_{L_{alg}^{<0}P} C^+$.
The class of $(g,\monx)$ in $L_{alg}^{<0}G\times_{L_{alg}^{<0}P} C^+$
is denoted by $[g :\monx]$.
Consider the map
$$
\begin{array}{cccc}
  \eta\,:&L_{alg}^{<0}G\times_{L_{alg}^{<0}P} C^+&\longto&\mvar\\
&[g:\monx]&\longmapsto&g\monx.
\end{array}
$$

Observe that $L_{alg}^{<0}G/{L_{alg}^{<0}P}$ and
$L_{alg}^{<0}G\times_{L_{alg}^{<0}P} C^+$ have no natural structure of
ind-varieties and are considered in this paper as sets. 

\begin{prop}
  \label{prop:etaGW}
Recall the definition of $n_\beta$ from  \eqref{eq:defnbeta}).
  \begin{enumerate}
  \item If $
  l(w_1)+l(w_2)+l(w_3)=\dim(G/P)+\sum_{\beta\in\Delta-\Delta_P}d_\beta n_\beta$
then for $\monx\in\mvar$ sufficiently general, the fiber $\eta^{-1}(\monx)$
has cardinality $GW(\sigma_{w_1},\sigma_{w_2},\sigma_{w_3};\ud)$.
\item If $
  l(w_1)+l(w_2)+l(w_3)=\dim(G/P)+\sum_{\beta\in\Delta-\Delta_P}d_\beta n_\beta$
then for $\monx\in\mvar$ sufficiently general, the fiber
$\eta^{-1}(\monx)$ is either empty or infinite.
  \end{enumerate}

\end{prop}

\begin{proof}
  Let $\monx=(g.L_0,g_1B/B,g_2B/B,g_3B/B)\in\mvar$ with $g\in
  L_{alg}G$, $g_i\in G$ for $i=1,2$ and $3$.
Since $\eta$ is $L^{<0}G$-equivariant, $L^{<0}G.L_0$ is dense in $\Gra$ and viewed the assumption of
genericity in the proposition, we may assume that $g$ is trivial.

If $\gamma$ denote an element of $L_{alg}^{<0}G/ L_{alg}^{<0}P$
(sometimes viewed as a curve on $G/P$ using $\Theta$), we denote by
$\tilde\gamma$ a representative in $L_{alg}^{<0}$.
Then $\eta^{-1}(\monx)$ identifies with 
$$
\{\gamma\in L_{alg}^{<0}G/ L_{alg}^{<0}P
\,|\,
\left\{
\begin{array}{l}
\tilde\gamma(p_i)^{-1}g_i B/B\in Pw_i^{-1}B/B\quad \forall i=1,2,3\\[3pt]
\tilde\gamma^{-1}\in (L_{alg}P)_0L_{-h}L_{alg}^{>0}G
\end{array}\right . 
\},
$$
that is with
$$
\{\tilde\gamma\in L_{alg}^{<0}G\,|\,
\left\{\begin{array}{l}
\tilde\gamma(p_i)\in g_iB w_iP,\quad \forall i=1,2,3\\[3pt]
\tilde\gamma\in  L_{alg}^{>0}Gh (L_{alg}P)_0
\end{array}
\right .
\}/ L_{alg}^{<0}P.
$$ 
By Lemma~\ref{lem:degdec}, this set identifies using $\Theta$ with
the set of curves $\gamma\in \Mor(\PP^1,G/P,\ud)$ such that
$\gamma(p_i)\in g_iB w_iP/P$, for any $i=1,2,3$. 
Now, the proposition follows from  Kleiman's theorem.
\end{proof}

\subsection{Other fibers of $\eta$}

Let $\underline{o}$ denote the base point of $L_{alg}G/
(L_{alg}P)_0$. Since $L^{<0}P$ is contained in $(L_{alg}P)_0$,
$L_{alg}^{<0}G.\underline{o}$ identifies with $L_{alg}G/
L^{<0}_{alg}P$.

  \begin{prop}
    \label{prop:fibrespe}
Let $l_1$, $l_2$ and $l_3$ in $L$.  
Set
$\monx=(L_{-h},\,l_1w_1^{-1}B/B,\,l_2w_2^{-1}B/B,\,l_3w_3^{-1}B/B)$ in
$C$.
Then the fiber $\eta^{-1}(\monx)$ identifies with the set of $g.\underline{o}\in L_{alg}^{<0}G/
L_{alg}^{<0}P$ such that 
\begin{enumerate}
\item $g\in z^{-h}L^{>0}_{alg}Gz^h\cap L^{<0}_{alg}G$ ;
\item $\Theta(g\underline{o})(p_i)\in l_iw_i^{-1}Bw_iP/P$ for any $i=1,2,3$.
\end{enumerate}
  \end{prop}

  \begin{proof}
    Like in the proof of Proposition~\ref{prop:etaGW}, we obtain that 
$\eta^{-1}(\monx)$ identifies with
$$
\{\tilde\gamma\in  L_{alg}^{<0}G\,:\,\left\{\begin{array}{l}
\tilde\gamma(p_i)\in l_iw_i^{-1}B w_iP,\quad \forall i=1,2,3\\[3pt]
\tilde\gamma\in z^{-h}L_{alg}^{>0}Gz^{h}L_{alg}P_0
\end{array}
\right .
\}/L^{<0}_{alg}P.
$$
In particular, 
$\tilde\gamma \underline{o}$ belongs to $L_{alg}^{<0}G.\underline{o}$
and to $z^{-h}L_{alg}^{>0}Gz^{h}.\underline{o}$. Since
$L_{alg}^{<0}G$, $z^{-h}L_{alg}^{>0}Gz^{h}$, and the stabilizer of
$\underline{o}$  contain $T$ the intersection
$L_{alg}^{<0}G.\underline{o} \cap
z^{-h}L_{alg}^{>0}Gz^{h}.\underline{o}$ is equal to
$(L_{alg}^{<0}G\cap z^{-h}L_{alg}^{>0}Gz^{h}) .\underline{o}$.
The proposition follows.
  \end{proof}


\section{Description of the GIT-cone}

\subsection{Satisfied inequalities}

\begin{lemma}
  \label{lem:ineqsat}
Let $(\lambda_1,\lambda_2,\lambda_3)\in(X^*(T)^+_\QQ)^3$ and $\bl\in
 \ZZ_{>0}$ such that $\frac{\lambda_1}\bl$, $\frac{\lambda_2}\bl$ and $\frac{\lambda_3}\bl$ belong to the
  alcove $\alc^*$. 
Let $\tau$ be a dominant one parameter subgroup of $T$ and set
$P=P(\tau)$.
Let  $w_1$, $w_2$, $w_3$ in $W^P$ and let $\ud=\sum_{\beta\in\Delta-\Delta_P}d_\beta\sigma_{s_\beta}^*$
for some $d_\beta\in \ZZ_{\geq 0}$. 
Set $h=\sum_{\beta\in\Delta-\Delta_P}d_\beta\beta^\vee\in X_*(T)$.
Assume that
$$
GW(w_1,w_2,w_3;\ud)\neq 0.
$$
If $(\lambda_1,\lambda_2,\lambda_3,\bl)\in \cone^{\rm nss}(\mvar)$ then
\begin{eqnarray}
 \label{eq:insat}
   \langle w_1\tau,\lambda_1\rangle+
\langle w_2\tau,\lambda_2\rangle+
\langle w_3\tau,\lambda_3\rangle\leq\bl\langle\tau,h\rangle.
\end{eqnarray}
\end{lemma}

\begin{proof}
Consider   $$
C^+=(L_{alg}P)_0L_{-h}\times Pw_1^{-1}B/B \times Pw_2^{-1}B/B \times Pw_3^{-1}B/B,
$$
and  the map
$$
\begin{array}{cccc}
  \eta\,:&L_{alg}^{<0}G\times_{L_{alg}^{<0}P} C^+&\longto&\mvar\\
&[g:\monx]&\longmapsto&g\monx.
\end{array}
$$

By Proposition~\ref{prop:etaGW} and Lemma~\ref{lem:ssopen}, there
exists a numericaly semistable point in the image of $\eta$. 
Then  there exists  a numericaly semistable point $\monx$ in $C^+$.
We deduce that $\mu^{\Li}(\monx,\tau)\leq 0$. 
By Lemma~\ref{lem:calculmu}, this inequality is equivalent to the
inequality to prove.
\end{proof}

\subsection{A first description of $\cone^{\rm nss}(\mvar)$}

We first reprove Teleman-Woodward's Theorem~\ref{th:TW} in our context.

\begin{lemma}\label{lem:TW}
  Let $(\lambda_1,\lambda_2,\lambda_3)\in(X^*(T)^+_\QQ)^3$ and $\bl\in
 \ZZ_{>0}$ such that $\frac{\lambda_1}\bl$, $\frac{\lambda_2}\bl$ and $\frac{\lambda_3}\bl$ belong to the
  alcove $\alc^*$. 
Then $(\lambda_1,\lambda_2,\lambda_3,\bl)\in \cone^{\rm nss}(\mvar)$  if and only
if
\begin{eqnarray}
   \langle w_1\varpi_{\beta^\vee},\lambda_1\rangle+
\langle w_2\varpi_{\beta^\vee},\lambda_2\rangle+
\langle w_3\varpi_{\beta^\vee},\lambda_3\rangle\leq\frac 2
{(\beta,\beta)} \bl d,
\end{eqnarray}
 for any simple root $\beta$,
any nonnegative integer $d$ and any $(w_1,w_2,w_3)\in (W^{P_\beta})^3$ such that 
\begin{eqnarray}
GW(w_1,w_2,w_3;d\sigma_{s_\beta}^*)=1.
\end{eqnarray}
\end{lemma}

\begin{proof}
If   $(\lambda_1,\lambda_2,\lambda_3,\bl)\in \cone^{\rm nss}(\mvar)$
then the inequalities are satisfied by Lemma~\ref{lem:ineqsat}. 
Conversely assume that $(\lambda_1,\lambda_2,\lambda_3,\bl)\not\in
\cone^{\rm nss}(\mvar)$ that is that $\mvar^{\rm nss}(\Li)$ is empty.
Consider the open stratum $\mvar^\circ(\Li)$ and $d_0$ the common
value of $M^\Li(\monx)$ for $\monx$ in $\mvar^\circ(\Li)$.
Let $\tau_0$, $P=P(\tau_0)$ and $C^+$ be like in
Proposition~\ref{prop:openstratum}. 
Write 
$$C=(L_{alg}L^{ss}.L_{-h},L.w_1^{-1}B/B,L.w_2^{-1}B/B,L.w_3^{-1}B/B),$$ 
with usual notation.
Write $h=\sum_{\beta\in\Delta-\Delta_P}d_\beta\beta^\vee$ and set 
 $\ud=\sum_{\beta\in\Delta-\Delta_P}d_\beta\sigma_{s_\beta}^*$.
Consider the map
$$
\begin{array}{cccc}
  \eta\,:&L_{alg}^{<0}G\times_{L_{alg}^{<0}P} C^+&\longto&\mvar\\
&[g:\monx]&\longmapsto&g\monx.
\end{array}
$$
By Proposition~\ref{prop:openstratum}, for any $\monx\in
\mvar^\circ(\Li)$, the fiber $\eta^{-1}(\monx)$ is not empty.
By Proposition~\ref{prop:unicitytau}, this fiber is reduced to one
point. 
Since $\mvar^\circ(\Li)$ is open, Proposition~\ref{prop:etaGW} implies
that 
$$
GW(w_1,w_2,w_3;\,\ud)=1.
$$
Lemma~\ref{lem:ineqsat} shows that inequality~\eqref{eq:insat} is
satisfied for any $\tau$ such that $P=P(\tau)$ and any point in
$\cone^{\rm nss}(\mvar)$. 
By construction the $\mvar^\circ(\Li)\cap C^+$ is not empty. Fix
$\monx$ in it. Then $\mu^\Li(\monx,\tau_0)=d_0>0$. Hence
Lemma~\ref{lem:calculmu} implies that  
inequality~\eqref{eq:insat} for $\tau=\tau_0$ is not satisfied by
$\Li$.\\

 With Lemma~\ref{lem:ineqsat}, we just proved that a point belongs to
 $\cone^{\rm nss}(\mvar)$ if and only if it satisfies the
 inequalities~\eqref{eq:insat} for any $\tau$, $h$ and $w_i$'s such
 that $
GW(w_1,w_2,w_3;\,\ud)=1.
$
It remains to prove that the inequalities coming from nonmaximal
parabolic subgroups are redundant. 
Consider such an inequality~\eqref{eq:insat} associated to some
non-maximal standard
parabolic subgroup $P$, some $\tau\in X_*(T)$,  and $w_1,w_2,w_3$ and
$h$.
Dualy, we have to prove that this  inequality~\eqref{eq:insat} does not generate an
extremal ray of the dual cone of $\cone^{\rm nss}(\mvar)$. 
By Lemma~\ref{lem:ineqsat},   inequality~\eqref{eq:insat}  holds for
any $\tau'\in X_*(T)$ such that $P=P(\tau')$. 
But, the set $\tau'$ such that $P=P(\tau')$ generate an open cone of
dimension two in $X_*(T)_\QQ$ and inequality~\eqref{eq:insat} depends
linearly on $\tau'$. Hence  inequality~\eqref{eq:insat} cannot be
extremal.
\end{proof}

\subsection{End of the proof of Theorem~\ref{th:mainGIT}}

\begin{proof}
It remains to prove that if
$(\lambda_1,\lambda_2,\lambda_3,\bl)\not\in \cone^{\rm nss}(\mvar)$,
then there exists an inequality~\eqref{eq:mainGIT} that satisfies
condition~\eqref{eq:filtrationGIT} and that is not fullfilled by this
point.
Consider the open strata
$\mvar^\circ(\Li)$. 
 Let $\tau_0$, $P=P(\tau_0)$ and $C^+$ be like in
Proposition~\ref{prop:openstratum}. 
Let $(L_{-h},w_1^{-1}B/B,w_2^{-1}B/B,w_3^{-1}B/B)\in C^+$ with usual
notation.
Let $h_{PW}\in X_*(T)$ be the Peterson-Woodward lifting of $h$. 
Let $\monx\in C^+\cap \mvar^\circ(\Li)$ and set $\monz=\lim_{t\to
  0}\tau(t)\monx$. 
By Proposition~\ref{prop:adapLss}, $\monz$ is numericaly semistable 
for the group $L_{alg}^{<0}L$ relatively to the line
  bundle $\Li\otimes
  -\mu^\Li(\monx,\tau_0)\frac{(\tau_0,\square)}{\Vert\tau_0\Vert}$.
In particular, $C^{\rm nss}(\Li\otimes
  -\mu^\Li(\monx,\tau_0)\frac{(\tau_0,\square)}{\Vert\tau_0\Vert},L_{alg}^{<0}L)$
  is not empty and open by Lemma~\ref{lem:ssopen}.
Now, Proposition~\ref{prop:adapLss} implies that for general $\mony$
in $C$, we have $\Vert\tau_0\Vert M^\Li(\mony)=\mu^\Li(\monx,\tau_0)$.
Then, Proposition~\ref{prop:unicitytau} shows that $\eta^{-1}(\mony)$
is one point for general $\mony$ in $C$.\\

The Peterson-Woodward lifting $h_{PW}$ has the property that the
$L_{alg}^{<0}L$-orbit of $L_{-h_{PW}}$ is dense in $L_{alg}L^{ss}.L_{-h}$. 
Since $\eta$ is equivariant, we deduce that for general $l_1$, $l_2$
and $l_3$ in $L$ the fiber
$\eta^{-1}(L_{-h_{PW}},\,l_1w_1^{-1}B/B,\,l_2w_2^{-1}B/B,\,l_3w_3^{-1}B/B)$
is one point.\\

Let $P^-$ denote the parabolic subgroup containing $T$ and opposite to
$P$. Consider the Lie algebra $Lie(R^u(P^-))$ of the unipotent radical
of $P^-$. It is a $L$-module. Consider its decomposition in weight
spaces under the action of $Z$:
\begin{eqnarray}
  \label{eq:decLieP}
  Lie(R^u(P^-))=\bigoplus_{\chi\in X^*(Z)}Lie(R^u(P^-))_\chi.
\end{eqnarray}
It is known that each $Lie(R^u(P^-))_\chi$ is an irreducible
$L$-module.

Consider the open $P^-$-orbit $\Omega$ in $G/P$. It is 
stable by the action of $L$ and isomorphic as a $L$-variety to
$Lie(R^u(P^-))$.
Let us fix such an
isomorphism $\zeta\,:\,Lie(R^u(P^-))\longto \Omega$.
For each $i=1,2,3$, set $V_i=\zeta^{-1}(\Omega\cap
w_i^{-1}Bw_iP/P)$. 
It is well known that $V_i$ is a linear subspace of
$Lie(R^u(P^-))$ stable by $Z$. Then 
\begin{eqnarray}
  \label{eq:decV}
  V_i=\bigoplus_{\chi\in X^*(Z)} V_{i,\chi},{\rm\,where}
\quad V_{i,\chi}=V_i\cap Lie(R^u(P^-))_\chi).
\end{eqnarray}

\bigskip
Set $K_h=z^{-h_{PW}}L^{>0}_{alg}Gz^{h_{PW}}\cap L^{<0}_{alg}G$. It is
a finite dimensional connected algebraic group containing $T$. Consider
$$
\mcM_h= K_h/(L_{alg}^{<0}P\cap K_h).
$$
Moreover, $\mcM$ contains 
$\mcM^\circ_h= K_h\cap L_{alg} R^u(P^-)$ as an open subset. 
But $\mcM_h$ is contained in
$L^{<0}_{alg}G/L_{alg}^{<0}P$ and  any $m\in\mcM_h$ can be seen as a
regular map $\Theta(m)$ from $\PP^1$ to $G/P$.
If $m$ belongs to $\mcM^\circ_h$, then $\Theta(m)(\PP^1-\{0\})$ is
contained in $\Omega$. Hence $\mcM^\circ_h$ can be seen as a set of
polynomial functions from $\PP^1-\{0\}$ to $\Omega$. 
Composing with $\zeta$ and identifying $\PP^1-\{0\}$ with $\CC$ (with
the coordiante $z^{-1}$), we get an embedding of $\mcM^\circ_h$ in ${\rm
  Mor}(\CC,Lie(R^u(P^-)))$.
For each root $\alpha\in \Phi(G/P)$, let us fix a nonzero element
$\xi_{-\alpha}\in Lie(R^u(P^-))_{-\alpha}$ of weight $-\alpha$.
The roots of $z^{-h_{PW}}L^{>0}_{alg}Gz^{h_{PW}}$ are the images of
the roots of $L^{>0}_{alg}G$ by $-h_{PW}$ viewed as an element of the
affine Weyl group (see Section~\ref{sec:Waff}). 
The root $\alpha+n\delta$ is a root of $K_h$ if and only if
 $-\langle h_{PW},\alpha\rangle\leq n\leq 0$.
Hence
$$
\mcM^\circ_h:=\{\sum_{\alpha\in \Phi(G/P)}P_\alpha\xi_{-\alpha}\,:\,
P_\alpha\in\CC[z^{-1}] \mbox{ and } \deg(P_\alpha)\leq  \langle h_{PW},\alpha\rangle\}
$$
has a natural structure of a complex vector space.
Note that, by convention, $\deg(0)=-\infty$. 

Consider now
\begin{eqnarray}
  \label{eq:defMchi}
  \mcM^\circ_{h,\chi}=\{\sum_{\alpha\in \Phi(G/P,\chi)}P_\alpha\xi_{-\alpha}\,:\,
P_\alpha\in\CC[z^{-1}] \mbox{ and } \deg(P_\alpha)\leq  \langle h_{PW},\alpha\rangle\}. 
\end{eqnarray}
Then $\mcM^\circ$  is a product:
\begin{eqnarray}
  \label{eq:decM}
  \mcM^\circ_h=\bigoplus_{\chi\in X^*(Z)}\mcM^0_{h,\chi}.
\end{eqnarray}
For any linear subspace $W$ in $Lie(R^u(P^-))$ and any $p\in \PP^1-\{0\}$, set
$$
\mcM^\circ_h(p,W):=\{m\in \mcM^\circ_h\,:\, m(p)\in W\}.
$$
Since $m\longmapsto m(p)$ is linear, $\mcM^\circ_h(p,W)$ is a linear
subspace and 
\begin{eqnarray}
  \label{eq:ineqdim78}
  \dim(\mcM^\circ_h)-\dim(\mcM^\circ_h(p,W))\leq \dim(Lie(R^u(P^-)))-\dim(W).
\end{eqnarray}
Similarly, for any $\chi\in X^*(Z)$, for any linear subspace $W$ in
$Lie(R^u(P^-))_\chi$  and any $p\in \PP^1-\{0\}$, set
$$
\mcM^\circ_{h,\chi}(p,W):=\{m\in \mcM^\circ_{h,\chi}\,:\, m(p)\in W\}.
$$
Then
\begin{eqnarray}
  \label{eq:ineqdimMchiW}
  \dim(\mcM^\circ_{h,\chi})-\dim(\mcM^\circ_{h,\chi}(p,W))\leq 
\dim(Lie(R^u(P^-))_\chi)-\dim(W).
\end{eqnarray}

Consider
\begin{eqnarray}
  \label{eq:defChi}
  \Chi=\{(l_1,l_2,l_3,m)\in L^3\times 
  \mcM^\circ_{h}\,:\,
\forall i=1,2,3\quad m(p_i)\in l_iV_{i}\},
\end{eqnarray}
and the two projections 
\begin{center}
\begin{tikzpicture}
  \matrix (m) [matrix of math nodes,row sep=3em,column sep=2.5em,minimum width=2em]
  {&\Chi&\\
L^3&&\mcM^\circ_{h}. \\};
  \path[-stealth]  
    (m-1-2) edge node[left] {p}  (m-2-1) 
    (m-1-2) edge node[auto] {q}  (m-2-3);       
\end{tikzpicture} 
\end{center}

Fix any $l_1$, $l_2$
and $l_3$ in $L$. Then $m\in q(p^{-1}(l_1,l_2,l_3))$ if and only if,
for any $i=1,2,3$, $m(p_i)$ belongs to $l_iV_i$. In other word
$$
q(p^{-1}(l_1,l_2,l_3))=\bigcap_{i=1}^3 \mcM^\circ_{h}(p_i,l_iV_i).
$$
The decomposition~\eqref{eq:decLieP} is respected by the action of
$L$, the subsapces $V_i$ (see~\eqref{eq:decV}) and the vector space 
$\mcM^\circ_h$ (see~\ref{eq:decM}). Hence
$$
q(p^{-1}(l_1,l_2,l_3))\simeq\bigoplus_{\chi\in X^*(Z)}
\bigcap_{i=1}^3 \mcM^\circ_{h,\chi}(p_i,l_iV_{i,\chi}).
$$

Proposition~\ref{prop:fibrespe} implies that for general $l_1$, $l_2$
and $l_3$ in $L$,
$p^{-1}(l_1,l_2,l_3)$ is one point. Then, for any $\chi\in X^*(Z)$,
$$
\sum_{i=1}^3
\dim(\mcM^\circ_{h,\chi})-\dim(\mcM^\circ_{h,\chi}(p_i,l_iV_{i,\chi}))\geq 
\dim(\mcM^\circ_{h,\chi}).
$$
Combining with~\eqref{eq:ineqdimMchiW}, we obtain
\begin{eqnarray}
  \label{eq:dim10}
3 \dim(Lie(R^u(P^-))_\chi)-
\sum_{i=1}^3
\dim(V_{i,\chi})
\geq 
\dim(\mcM^\circ_{h,\chi}).
\end{eqnarray}
From~\eqref{eq:defMchi}, we deduce
\begin{eqnarray}
  \label{eq:dim11}
\dim(\mcM^\circ_{h,\chi})\geq \sum_{\alpha\in \Phi(G/P,\chi)}
\bigg(\langle h_{PW},\alpha\rangle+1\bigg),
\end{eqnarray}
and 
\begin{eqnarray}
  \label{eq:dim12}
3 \dim(Lie(R^u(P^-))_\chi)-
\sum_{i=1}^3
\dim(V_{i,\chi})
\geq 
\sum_{\alpha\in \Phi(G/P,\chi)}
\bigg(\langle h_{PW},\alpha\rangle+1\bigg).
\end{eqnarray}
By summing inequalities~\eqref{eq:dim12}, when $\chi$ runs in $X^*(Z)$,
we find
$$3\dim(G/P)-\sum_{i=1}^3\dim(V_i)\geq
\dim(G/P)+
\sum_{\alpha\in
    \Phi(G/P)}\langle h_{PW},\alpha\rangle.
$$
Since $GW(w_1,w_2,w_3;\underline{d})=1$, this inequality is actually
an equality. 
Hence
each inequalities~\eqref{eq:dim12} is an equality.
These equalities are readily equivalent to condition~\eqref{eq:filtrationGIT}.
\end{proof}

\bigskip
\begin{NB}
  The above proof shows that inequality~\eqref{eq:dim11} is an
  equality, in the setting of the theorem. With~\eqref{eq:defMchi}
  this implies that 
  \begin{eqnarray}
    \label{eq:conjcomb}
    \forall\alpha\in \Phi(G/P,\chi)\quad \langle
    h_{PW},\alpha\rangle\geq -1.
  \end{eqnarray}
It is a natural question to ask if inequality~\eqref{eq:conjcomb} is
satisfied for any maximal $P$ (associated to the simple root $\beta$)
and any $h\in\ZZ_{\geq 0}\beta^\vee$.
\end{NB}

\bibliographystyle{amsalpha}
\bibliography{qHorn}

\begin{center}
  -\hspace{1em}$\diamondsuit$\hspace{1em}-

\end{center}

\end{document}